\documentclass{lms}
\usepackage{amsmath}
\usepackage{amssymb}
\usepackage{amscd}
\usepackage[all]{xy}

\numberwithin{equation}{section}

\newcounter{AbcT}

\newtheorem {Theorem}    {Theorem}[section]
\newunnumbered {TheoremA}    {Theorem A.1}
\newunnumbered {Definition} {Definition} 

\newtheorem {Question}    [Theorem]{Question}

\newtheorem {Lemma}      [Theorem]    {Lemma}
\newtheorem {Corollary}  [Theorem]    {Corollary}
\newtheorem {Proposition}[Theorem]    {Proposition}
\newtheorem {Claim}      [Theorem]    {Claim}

\newcounter{DM@bibnum}

\newcommand {\F} {{\mathbb F}}

\newcommand{\la}{\langle}
\newcommand{\ra}{\rangle}

\def\KMS{{\rm KMS}}

\def\log{{\rm log\,}}

\def\Ker{{\rm Ker\,}}
\def\Im{{\rm Im\,}}

\def\LT{{\rm LT\,}}

\def\NSL_2{{\mathcal N SL_2}}

\def\sl{{\mathfrak{sl}}}

\def\Gp{{\rm Gr}}
\def\depth{{\rm depth}}

\def\eps{\varepsilon}

\def\lam{\lambda}            
                
\def\phi{\varphi}

\def\calB{{\mathcal B}}

\def\calU{{\mathcal U}}

\def\bgal{\widetilde b}

             \def\Hgal{{\widetilde H}}
             \def\Igal{{\widetilde I}}

\def\mgal{\tilde m}

\def\rgal{\widetilde r}             \def\Rgal{{\widetilde R}}
\def\sgal{\tilde s}             \def\Sgal{{\widetilde S}}
             
             \def\Ugal{{\widetilde U}}
             
\def\wgal{\widetilde w}             \def\Wgal{{\widetilde W}}
\def\xgal{\tilde x}             \def\Xgal{{\widetilde X}}

\def\fbar{\bar f}               \def\Fgag{\overline{\strut F}}
               
\def\hbar{\bar h}

               \def\Wgag{\overline{\strut W}}

               \def\Ghat{\widehat {\mathstrut G}}

\def\phat{\widehat p}

\def\dbF{{\mathbb F}}

\def\dbN{{\mathbb N}}

\def\dbR{{\mathbb R}}

\def\dbZ{{\mathbb Z}}

\def\Fp{{\dbF_p}}

\def\Fq{{\dbF_q}}

\newcommand{\lla}{\la\!\la}
\newcommand{\rra}{\ra\!\ra}

\begin{document}

\title[Kazhdan quotients of Golod-Shafarevich groups]
{Kazhdan quotients of Golod-Shafarevich groups}
\author{Mikhail Ershov}
\classno{Primary 20F05, 20F69, Secondary 20E18, 20F65, 22D10, 43A07}
\maketitle

\vskip -.4cm
\centerline{ with an appendix on }
\vskip .1cm
\centerline{\bf Uniform non-amenability of Golod-Shafarevich groups}
\vskip .2cm
\centerline{by {\sc Mikhail Ershov} and {\sc Andrei Jaikin-Zapirain}}
\vskip .3cm
\centerline{ and an appendix on }
\vskip .1cm
\centerline{\bf Subgroup growth of Golod-Shafarevich groups}
\vskip .2cm
\centerline{by \sc Andrei Jaikin-Zapirain}
\vskip .3cm

\begin{abstract}
The main goal of this paper is to prove that every Golod-Shafarevich group
has an infinite quotient with Kazhdan's property~$(T)$.
In particular, this gives an affirmative answer to the well-known
question about non-amenability of Golod-Shafarevich groups.
\end{abstract}

\section{Introduction}
\subsection{Golod-Shafarevich groups}
In early 60's, Golod and Shafarevich~\cite{GSh} found a sufficient
condition for a group given by generators and relators
to be infinite. The original condition from \cite{GSh} admits
several generalizations, but the following one
is usually taken as the definition of Golod-Shafarevich groups.
\begin{Definition}\rm Fix a prime number $p$.
\begin{enumerate}
\item[(a)] Consider a group presentation $\la X|R\ra$, where $X$ is finite.
For each $i\in\dbN$ let $r_i$ be the number of elements of $R$
which have degree $i$ with respect to the Zassenhaus $p$-filtration.
The presentation $\la X|R\ra$ is said to satisfy the 
\textit{Golod-Shafarevich} (GS)
condition with respect to $p$, if there exists a real number $t\in (0,1)$ such that
$$1-H_X(t)+H_R(t)<0 \mbox { where } H_X(t)=|X|t\mbox{ and }H_R(t)=\sum_{i=1}^{\infty}r_i t^i.$$

\item[(b)] A group $\Gamma$ is called \textit{Golod-Shafarevich} if it
has a presentation satisfying the Golod-Shafarevich condition.
\end{enumerate}
\end{Definition}
\begin{Remark} The reference to prime $p$ will usually be omitted, as we will
never consider GS groups with respect to different primes at the same time.
\end{Remark}
\vskip .3cm
The Golod-Shafarevich condition can also be defined for pro-$p$ groups
(using essentially the same definition) and certain kinds of associative
algebras (with suitable notion of degree) -- see Section~2 for details.
In fact, to prove that any (abstract)\footnote{We shall often refer
to ordinary groups as `abstract groups' to distinguish them from pro-$p$ groups}
Golod-Shafarevich group is infinite, one reduces the problem to the corresponding question for algebras
defined as quotients of $K\lla u_1,\ldots, u_n\rra$, where $K\lla u_1,\ldots, u_n\rra$
is the algebra of non-commutative power series in $u_1,\ldots,u_n$ over a field $K$.
More specifically, given a finitely generated group $\Gamma$, one considers the pro-$p$
completion $\Gamma_{\phat}$ of $\Gamma$
and its completed group algebra $\Fp[[\Gamma_{\phat}]]$. Every presentation
$\la X|R \ra$ for $\Gamma$ yields the corresponding presentation for $\Fp[[\Gamma_{\phat}]]$
as a quotient of $\Fp\lla u_1,\ldots, u_n\rra$ where $n=|X|$, and
a clever dimension-counting argument shows that $\Fp[[\Gamma_{\phat}]]$
is infinite provided $\la X|R \ra$ satisfies the GS condition. If the algebra $\Fp[[\Gamma_{\phat}]]$
is infinite, the groups $\Gamma_{\phat}$ and $\Gamma$ must also be infinite.
\vskip .12cm

Besides being infinite, Golod-Shafarevich (GS) groups are known to be ``large''
in various senses: for instance, any Golod-Shafarevich group $\Gamma$
has an infinite torsion quotient (see \cite{Go} and \cite{Wi}), the pro-$p$ completion of $\Gamma$ contains a non-abelian free pro-$p$ group (see \cite{Ze}), and if $\{\omega_n \Gamma\}$ is the
Zassenhaus $p$-series of $\Gamma$, the sequence $\{\log_p |\Gamma/\omega_n \Gamma|\}_{n=1}^{\infty}$
has exponential growth (see discussion in Section~2).

\subsection{The main theorem and a naive approach to it}
In \cite{Er}, the author constructed the first examples of Golod-Shafarevich
groups with Kazhdan's property $(T)$. Following this discovery, Lubotzky proposed a
related conjecture that every Golod-Shafarevich group should have an infinite quotient with
property $(T)$.
The main goal of this paper is to prove Lubotzky's conjecture:
\begin{Theorem}
\label{thm:main0}
Every Golod-Shafarevich group has an infinite quotient with Kazhdan's property $(T)$.
\end{Theorem}
The result of \cite{Er} was quite surprising and seemed to go against the general theme of ``largeness'' of 
Golod-Shafarevich groups. On the other hand, Theorem~\ref{thm:main0} can be considered a natural addition to the list of ``largeness''
properties of Golod-Shafarevich groups. 
Another important consequence of this theorem
is an affirmative answer to a well-known question about non-amenability of Golod-Shafarevich groups
(see, e.g., \cite[Open Problem~5.2]{Har}),
since quotients of amenable groups are amenable,
while an infinite discrete group cannot be both amenable and Kazhdan.

\begin{Corollary}
\label{thm:nonam}
Golod-Shafarevich groups cannot be amenable.
\end{Corollary}

In fact, in Appendix~A to this paper it will be shown that Golod-Shafarevich groups
satisfy a stronger form of non-amenability, called \textit{uniform non-amenability}.

Before discussing the proof of Theorem~\ref{thm:main0}, we observe
that the analogous result is known to be true for (non-elementary) hyperbolic groups,
that is, every hyperbolic group has an infinite quotient with property $(T)$.
The latter is a combination of the following two deep results and
the fact that property $(T)$ is preserved by quotients:
\begin{itemize}
\item[(a)] There exist hyperbolic groups with property $(T)$.
\item[(b)] Any two hyperbolic groups have a common infinite quotient.
\end{itemize}

Since Golod-Shafarevich groups with $(T)$ exist by \cite{Er},
a ``naive'' approach to Theorem~\ref{thm:main0} would be to try
to prove the analogue of (b) for Golod-Shafarevich groups (with respect
to a fixed prime $p$), but such an assertion is almost definitely false.

In fact, the following
would be sufficient to prove Theorem~\ref{thm:main0}.
Suppose that given a Golod-Shafarevich group $\Gamma$ and
a presentation $\la X | R\ra$ for $\Gamma$ with
$1-H_{X}(t)+H_{R}(t)<0$ for some $t=t(\Gamma)\in (0,1)$,
we can construct a group $\Gamma'$ with property $(T)$
which has a presentation $\la X | R'\ra$ (with the same generating
set $X$ as $\Gamma$) such that
$H_{R'}(t)<-(1-H_{X}(t)+H_{R}(t))$ for $t=t(\Gamma)$
(so, in particular, $\Gamma_1$ is also a Golod-Shafarevich group).
Then the group $\Gamma''=\la X\mid R\cup R'\ra$ is a common quotient of $\Gamma$ and $\Gamma'$,
and $\Gamma''$ is Golod-Shafarevich (hence infinite) since
$1-H_{X}(t)+H_{R\cup R'}(t)\leq 1-H_{X}(t)+H_R(t)+H_{R'}(t)<0$.
Thus, $\Gamma$ has an infinite quotient $\Gamma''$, which has property $(T)$
being a quotient of $\Gamma'$.

The supply of Golod-Shafarevich groups with property $(T)$ given by \cite{Er}
was clearly too short for the above strategy to work for an arbitrary Golod-Shafarevich group $\Gamma$.
In \cite{EJ} a larger collection of Golod-Shafarevich groups with $(T)$
was constructed within the class of \textit{Kac-Moody-Steinberg} groups, also defined in \cite{EJ}.
This collection turns out to be sufficient for the proof of Theorem~\ref{thm:main0},
but the approach outlined in the previous paragraph cannot be applied directly.
\subsection{Outline of the proof of the main theorem.}

A key ingredient in the proof of Theorem~\ref{thm:main0} is analysis
of presentations for finite index subgroups of Golod-Shafarevich pro-$p$ groups.
It is easy to see that if $G$ is a GS group (abstract or pro-$p$), then standard presentations
for finite index subgroups of $G$ (obtained via Schreier rewriting process) 
do not necessarily satisfy the GS condition. However, as we will show in Section~3, finite index subgroups of
Golod-Shafarevich pro-$p$ groups always satisfy the generalized Golod-Shafarevich (GGS) condition.
Similarly to the usual GS condition, the generalized Golod-Shafarevich condition
for a presentation $\la X | R\ra$ is defined by
inequality of the form $1-H_{D,X}(t)+H_{D,R}(t)<0$, where
$H_{D,X}(t)$ and $H_{D,R}(t)$ generalize the previously defined series
$H_X(t)$ and $H_R(t)$. This time we start with a
degree function $D:X\to \dbN$ on the generating set $X$, then extend $D$ in a
canonical way to the free pro-$p$ group on $X$ and put $H_{D,X}(t)=\sum_{x\in X}t^{D(x)}$
and $H_{D,R}(t)=\sum_{r\in R}t^{D(r)}$ (the usual GS inequality corresponds to the
case $D(x)=1$ for all $x\in X$).

The key result of Section~3 (see Theorem~\ref{deepdescent3})
asserts that if $G$ is any GGS pro-$p$ group,
then for any real number $M>0$ there exists a finite index subgroup
$K$ of $G$, a (pro-$p$) presentation $\la X| R\ra$ of $K$, a degree function $D$ on $F(X)$,
and a real number $t_0\in (0,1)$ such that $1-H_{D,X}(t_0)+H_{D,R}(t_0)<-M$.

Now given a GGS (abstract) group $\Gamma$, one can prove the existence of an infinite quotient with $(T)$
for $\Gamma$ using the following three-step algorithm:
\begin{itemize}
\item[(1)] Let $\Gamma_{\phat}$ be the pro-$p$ completion of $\Gamma$ (then
$\Gamma_{\phat}$ is a GGS pro-$p$ group).
Find a finite index subgroup $K$ of $\Gamma_{\phat}$ which
has a presentaiton $\la X| R\ra$ satisfying $$1-H_{D,X}(t_0)+H_{D,R}(t_0)<-6\cdot 10^4$$
for some $t_0\in (0,1)$ and degree function $D$.

\item[(2)] (see Theorem~\ref{jaikin4} and a remark after it)
Let  $\Delta$ be a dense finitely generated subgroup of $K$.
Show that there is a Kac-Moody-Steinberg group
$\Lambda$ with property $(T)$ such that
$\Delta$ and $\Lambda$ have a common infinite quotient $\Omega$.
Note that we can take $\Delta=K\cap\iota(\Gamma)$ (where
$\iota:\Gamma\to \Gamma_{\phat}$ is the canonical map),
and thus a finite index subgroup of $\Gamma$ (namely $\iota^{-1}(K)\cap \Gamma$)
has an infinite quotient with $(T)$.

\item[(3)] Deduce from Step~(2) that the entire group $\Gamma$
has an infinite quotient with $(T)$.
\end{itemize}

Of course, Step~(1) can be accomplished by Theorem~\ref{deepdescent3}, and
Step~(3) is possible thanks to the following general result due to Andrei Jaikin-Zapirain
(see Theorem~\ref{finiteindex}):
If $\Gamma$ is any finitely generated group such that
some finite index subgroup of $\Gamma$
has an infinite quotient with $(T)$, then $\Gamma$ itself has
an infinite quotient with $(T)$.

We now comment on how to construct the group $\Omega$ in Step~(2)
and show that it is infinite. First we construct another presentation $\la X'| R'\ra$ of the group $K$
(from Step~(1)) and a new degree function $D'$ such that $H_{D',X'}(t'_0)=12$,
$H_{D',R'}(t'_0)<\frac{1}{1000}$ for some $t'_0\in (0,1)$, and an extra technical condition holds
(this is possible by Theorem~\ref{deepdescent4}).
If $p\geq 67$, we construct $\Omega$ by adapting the ``naive'' approach to Theorem~\ref{thm:main0}
discussed earlier, and we are able to make $\Omega$ a GGS group (and hence infinite).

If $p<67$, this technique does not apply directly; in fact, we do not know any examples of
groups with $(T)$ which are GGS with respect to $p<67$.
In this case we prove that $\Omega$ is infinite by showing that
there exist a finite field $\dbF$ of characteristic $p$
and an $\dbF$-algebra $A$ which is GGS as an algebra over $\dbF$
(hence infinite), such that the completed group algebra $\Fp[[{\Omega}_{\phat}]]$
can be mapped onto a finite codimension subalgebra of $A$. Note that neither
$\Fp[[{\Omega}_{\phat}]]$ nor $A$ will be GGS as algebras over $\Fp$.
What makes such phenomenon possible is that many relations
needed to define $A$ as an algebra over $\Fp$ become
redundant if the ground field is changed from $\Fp$ to $\dbF$.

\subsection{Organization of the paper} In Section~2
we set up basic notations and discuss various forms of the Golod-Shafarevich condition
for abstract groups, pro-$p$ groups and associative algebras. In Section~3
we analyze presentations for subgroups of generalized Golod-Shafarevich groups.
Finally, Section~4 contains the proof of Theorem~\ref{thm:main0}.

The paper concludes with two appendices. Appendix~A by the author and Andrei Jaikin-Zapirain
contains a proof of uniform non-amenability of Golod-Shafarevich groups.
Basic background on Kazhdan's property (T) and amenability is given at the
beginning of Appendix~A. In Appendix~B by Andrei Jaikin-Zapirain
a new lower bound for the subgroup growth of Golod-Shafarevich groups
is established.

\vskip .2cm
{\bf Acknowledgements.} I am extremely grateful to Andrei Jaikin-Zapirain whose
generous contributions to this manuscript yielded significant improvement of the
main results and helped clarify the exposition. In particular, the proofs of Theorem~\ref{finiteindex} and
Proposition~\ref{deepdescent2} and generalization of many results in Section~3
are all due to him.
I would also like to thank Efim Zelmanov for introducing me to Golod-Shafarevich groups
several years ago and Martin Kassabov and Alex Lubotzky for useful discussions at various stages
of this project. 
This work is supported in part by the NSF grant DMS-0901703.

\section{Pro-$p$ presentations and Golod-Shafarevich theorem}

While the main result of this paper is about abstract groups,
we will work primarily with presentations
of finitely generated pro-$p$ groups. For this reason,
the simplest terminology and notations will be reserved for pro-$p$ groups
and their presentations.

Throughout the paper $p$ will be a fixed prime number.
Given a finite set $X$, by $F(X)$ we denote the free pro-$p$ group on $X$.
If $F$ is a finitely generated free pro-$p$ group and $X$ is a free
generating set for $F$, we shall canonically identify $F$ with $F(X)$.

A \textit{presentation} is a pair $(X,R)$,  where
$X$ is a finite set and $R$ is a countable subset of $F(X)$.
By $\Gp(X,R)$ we will denote the pro-$p$ group
defined by the presentation $(X,R)$, that is,
$$\Gp(X,R)=F(X)/{\la R\ra}^{F(X)},$$
where ${\la R\ra}^{F(X)}$ is the (closed) normal subgroup of $F(X)$ generated by $R$.

The free abstract group on $X$ will be denoted by $F_{abs}(X)$, and
ordinary presentations of abstract groups will be referred to as
\textit{abstract presentations.} 

\vskip .1cm
Given a finite set $U=\{u_1,\ldots, u_n\}$ and a field $K$,
we put $K\lla U\rra=K\lla u_1,\ldots, u_n\rra,$
the algebra of non-commutative power series in $u_1,\ldots, u_n$ over $K$.
Presentations of pro-$p$
groups will be studied via the Magnus embedding
of a free pro-$p$ group $F(\{x_1,\ldots,x_n\})$
into $(\Fp\lla u_1,\ldots, u_n\rra)^*$ given by
$x_i\mapsto 1+u_i$ for $1\leq i\leq n$. Such use
of letters $X$ and $U$ will be kept throughout the paper.

\subsection{Degree and weight functions} We start by defining the notions of degree
and weight functions on algebras of the form $K\lla U\rra$.
\vskip -.2cm
\begin{Definition}\rm
Let $U=\{u_1,\ldots, u_n\}$ be a finite set and let $K$ be a field.

\noindent
(a) A function $d: K\lla U\rra\to \dbZ_{\geq 0}\cup\{\infty\}$ is called
a \textit{degree function} on $K\lla U\rra$ if
\begin{itemize}
\item[(i)] $d(f)=\infty$ if and only if $f=0$
\item[(ii)] $d(f)=0$ if and only if $f$ is invertible in $K\lla U\rra$
\item[(iii)] $d(fg)=d(f)+d(g)$ for any $f,g\in K\lla U\rra$
\item[(iv)] If $f=\sum_{\alpha} c_{\alpha} m_{\alpha}$ where each
$c_{\alpha}\in K$, each $m_{\alpha}$ is of the form $u_{i_1}\ldots u_{i_k}$
and all $m_{\alpha}$ are distinct,
then $d(f)=\min\{d(m_{\alpha}) : c_{\alpha}\neq 0\}$.
\end{itemize}
(b) A function $w: K\lla U\rra\to [0,1]$ is called
a \textit{weight function} on $K\lla U\rra$ if
\begin{itemize}
\item[(i)] $w(f)=0$ if and only if $f=0$
\item[(ii)] $w(f)=1$ if and only if $f$ is invertible in $K\lla U\rra$
\item[(iii)] $w(fg)=w(f)w(g)$ for any $f,g\in K\lla U\rra$
\item[(iv)] If $f=\sum_{\alpha} c_{\alpha} m_{\alpha}$ where each
$c_{\alpha}\in K$, each $m_{\alpha}$ is of the form $u_{i_1}\ldots u_{i_k}$
and all $m_{\alpha}$ are distinct,
then $w(f)=\max\{w(m_{\alpha}) : c_{\alpha}\neq 0\}$.
\end{itemize}
\end{Definition}
\vskip -.2cm
\begin{Remark} Every degree function $d$ on $K\lla U\rra$ satisfies the condition
\begin{itemize}
\item[(iv)'] $d(f+g)\geq \min\{d(f), d(g)\}$ for any $f,g\in K\lla U\rra$
\end{itemize}
(this follows from (iv)), so $d$ is a ring valuation on $K\lla U\rra$.
Similarly, for any weight function $w$ on $K\lla U\rra$ we have
\begin{itemize}
\item[(iv)''] $w(f+g)\leq \max\{w(f), w(g)\}$ for any $f,g\in K\lla U\rra$.
\end{itemize}
\end{Remark}
\vskip .15cm
Next we define the corresponding notions for free pro-$p$ groups.
Let $F$ be a finitely generated free pro-$p$ group and $X=\{x_1,\ldots,x_n\}$
a free generating set for $F$. Consider another set with $n$ elements
$U=\{u_1,\ldots, u_n\}$ and embed $F$ into $\Fp\lla U\rra^*$ by sending $x_i$ to $1+u_i$.

\begin{Definition}\rm Let $F$, $X$ and $U$ be as above.
\begin{itemize}
\item[(a)] Two functions $\Phi: F\to \dbR$ and $\phi:\Fp\lla U\rra\to\dbR$
will be called \textit{$X$-compatible} if
\begin{equation*}
\Phi(f)=\phi(f-1) \mbox{ for any } f\in F.
\end{equation*}
\item[(b)] A function $D: F\to \dbN\cup\{\infty\}$ is called
a \textit{degree function} on $(F,X)$ if it is $X$-compatible with
some degree function $d$ on $\Fp\lla U\rra$.
\item[(c)] The unique degree function $D$ on $(F,X)$
such that $D(x)=1$ for each $x\in X$ is called the \textit{standard degree function} on $(F,X)$.
\item[(d)] A function $W: F\to \dbZ_{\geq 0}\cup\{\infty\}$ is called
a \textit{weight function} on $(F,X)$ if it is $X$-compatible with
some weight function $w$ on $\Fp\lla U\rra$.
\end{itemize}
\end{Definition}
\vskip -.2cm

Clearly, there is a bijection between degree functions
on $(F,X)$ and $\Fp\lla U\rra$, and the same is true for
weight functions. From now on we shall be mostly concerned
with degree and weight functions on $(F,X)$.

It is easy to see that any weight function on $(F,X)$ is of the following form:
take any $t_1,\ldots, t_n\in (0,1)$, set $W(x_i)=t_i$,
and then extend $W$ uniquely to $F$ using (iii) and (iv)
in the definition of a weight function on $\Fp\lla U\rra$.
Any weight function can be approximated (in the suitable sense)
by integral weight functions defined below.

\begin{Definition}\rm Let $D$ be a degree function on $(F,X)$ and $t\in (0,1)$.
The function $W: F\to (0,1)$ defined by $$W(f)=t^{D(f)}$$ will be called
the \textit{$(D,t)$-weight function} on $(F,X)$. Weight functions
of this form will be called \textit{integral}.
\end{Definition}

\vskip .12cm
The following definition will be convenient for the study of Golod-Shafarevich groups.
\begin{Definition}\rm Let $X$ be a finite set and $S$ a countable subset of $F=F(X)$.
\begin{itemize}
\item[(a)] If $W$ is a weight function on $(F,X)$, we set $$W(S)=\sum_{f\in S} W(f)\in\dbR_{\geq 0}\cup\{\infty\}$$
\item[(b)] Let $D$ be a degree function on $(F,X)$, and assume that for each $n\in\dbN$
the set $\{s\in S\,:\, D(s)=n\}$ is finite. Then the power series
$$H_{D,S}(t)=\sum_{f\in S}t^{D(f)}$$
will be called the \textit{Hilbert series} of $S$ (with respect to $D$). We will write
$H_{S}$ instead of $H_{D,S}$ whenever $D$ is clear from context. Note that for any $t_0\in [0,1)$
we have $H_{D,S}(t_0)=W(S)$ where $W$ is the $(D,t_0)$-weight function on $(F,X)$.
\end{itemize}
\end{Definition}
Similarly, one defines weights and Hilbert series for subsets of power series algebras $K\lla U\rra$.
\vskip .12cm

Next we define a class of filtrations on finitely generated
pro-$p$ groups which come from degree functions.

\begin{Definition}\rm Let $G$ be a finitely generated pro-$p$ group,
$X$ a finite set and $\pi:F(X)\to G$ a surjective homomorphism.
Let $D$ be a degree function on $(F(X),X)$. We define the filtration $\{G_n\}$
of $G$ by
$$G_n=\{g\in G : g=\pi(f) \mbox{ for some }f\in F(X) \mbox{ with }D(f)\geq n.\}$$
We will call $\{G_n\}$ the $D$-filtration of $G$.
\end{Definition}
It is easy to see that each $G_n$ is an open subgroup of $G$ and $\cap \{G_n\}=\{1\}$.
Furthermore, $\{G_n\}$ is a $p$-filtration of $G$, that is,
$[G_i,G_j]\subseteq G_{i+j}$ and $(G_i)^p\subseteq G_{pi}$.
If $D$ is the standard degree function on $(F(X),X)$, the filtration $\{G_n\}$ is known as the
\textit{Zassenhaus $p$-filtration} of $G$.
\vskip .1cm
Similarly, if $d$ is a degree function on a power series algebra $K\lla U\rra$,
one defines the $d$-filtration on any algebra which is a quotient of $K\lla U\rra$.
\vskip .1cm

Finally, we recall the classical definition of a Hilbert series which can be associated
to any algebra with a suitable filtration.
\begin{Definition}\rm  Let $A$ be an associative algebra over some field $K$,
and let $A=A_0\supseteq A_1 \supseteq A_2\supseteq \ldots$ be a descending
chain of $K$-subalgebras in $A$ such that $A_i\cdot A_j\subseteq A_{i+j}$
and $\dim_{K}(A_i/A_{i+1})<\infty$. The series
$$Hilb_{A}(t)=\sum_{n=0}^{\infty}\dim_{K} (A_n/A_{n+1}) t^n$$
will be called the \textit{classical Hilbert series} of $A$ with respect to $\{A_n\}$.
\end{Definition}

\subsection{Golod-Shafarevich inequality}

Although the Golod-Shafarevich inequality is usually introduced as a tool for
proving that some groups are infinite, it is actually a result about
algebras. The following theorem which we call the \textit{generalized Golod-Shafarevich inequality }
(GGS inequality) is proved in Koch's Appendix to \cite{Ko2} (see formula (2.11) on p.105 of \cite{Ko2} 
and a remark at the end of this section).
\begin{Theorem}[(GGS inequality)]
\label{GGSineq}
Let $U$ be a finite set, $K$ a field and $d$ a degree function on $K\lla U\rra$.
Let $A=K\lla U\rra/I$ for some ideal $I$, and let $S$ be a generating set for $I$
such that $\{s\in S\colon d(s)=n\}$ is finite for each $n\in\dbN$.
\footnote{It is easy to see that any ideal of $K\lla U\rra$ has a generating set
with this property.}
Let $Hilb_A(t)$ be the classical Hilbert series of $A$ with respect to the $d$-filtration.
Then
$$\frac{(1-H_{d,U}(t)+H_{d,{S}}(t))\cdot Hilb_{A}(t)}{1-t}\geq \frac{1}{1-t},$$
where for power series $\sum \alpha_i t^i$ and $\sum \beta_i t^i$,
inequality $\sum \alpha_i t^i\geq\sum \beta_i t^i$ means that $\alpha_i\geq \beta_i$
for each $i$.
\end{Theorem}
The important consequence of the GGS inequality is that if $1-H_{d,U}(t_0)+H_{d,S}(t_0)<0$
for some $0<t_0<1$, the series $Hilb_{A}(t_0)$ must diverge, so in particular $A$ is infinite-dimensional.
We now explain how GGS inequality applies to (pro-$p$) groups.

Let $G$ be a pro-$p$ group given by a presentation $(X,R)$,
where $X=\{x_1,\ldots, x_n\}$. As before, consider the standard
embedding of $F(X)$ into $\Fp\lla U\rra$ where $U=\{u_1,\ldots, u_n\}$,
and let $I$ be the ideal of $\Fp\lla U\rra$ generated by the set $S=\{r-1 : r\in R\}$. It is well known
\cite[Theorem~7.17]{Ko} that there is a (canonical) isomorphism between
the completed group algebra $F_p[[G]]$ and $\Fp\lla U\rra/I$ making the following
diagram commutative:
\begin{equation}
\label{xy:GSdiag}
\xymatrix{
F(X)\ar[d]\ar[rr] &&  \Fp\lla U\rra  \ar[d]\\
G\ar[r]& \F_p[[G]] \ar[r] &  \Fp\lla U\rra/I
}
\end{equation}
(where all other maps are defined in the obvious way).

Now let $D$ and $d$ be $X$-compatible degree functions on $(F(X),X)$ and $\Fp\lla U\rra$, respectively,
so that $H_{D,X}(t)= H_{d,U}(t)$ and $H_{D,R}(t)= H_{d,S}(t)$.
Define the \textit{$D$-filtration on $\Fp[[G]]$} to be
the image of the $d$-filtration on $\Fp\lla U\rra/I$ under the isomorphism
$\Fp\lla U\rra/I\to \Fp[[G]]$ in \eqref{xy:GSdiag} and let
$Hilb_{D,\Fp[[G]]}(t)$ be the associated (classical) Hilbert series.
Then the inequality of power series in Theorem~\ref{GGSineq} can be rewritten
as follows:
\begin{equation}
\label{GGSg}
\frac{(1-H_{D,X}(t)+H_{D,{R}}(t))\cdot Hilb_{D,\Fp[[G]]}(t)}{1-t}\geq \frac{1}{1-t}
\end{equation}
Thus, if
$$1-H_{D,X}(t_0)+H_{D,R}(t_0)<0 \mbox{ for some } t_0\in (0,1),\eqno (GGS)$$
then $Hilb_{D,\Fp[[G]]}(t_0)$ diverges, so the algebra $\Fp[[G]]$
(and hence the group $G$) is infinite.

\begin{Definition}\rm (a) A presentation $(X,R)$ is said to satisfy
the \textit{generalized Golod-Shafarevich (GGS) condition} if there
exists a degree function $D$ on $(F(X),X)$ and $t_0\in (0,1)$ such that
$1-H_{D,X}(t_0)+H_{D,R}(t_0)<0$.

(b) A pro-$p$ (resp. an abstract) group is called a \textit{generalized Golod-Shafarevich group}
if it has a pro-$p$ (resp. abstract) presentation satisfying the GGS condition.

(c) A pro-$p$ (resp. an abstract) group is called a \textit{Golod-Shafarevich group}
if it has a pro-$p$ (resp. abstract) presentation satisfying the GGS condition
with respect to the standard degree function.
\end{Definition}
\vskip -.3cm
The following result summarizes the above discussion.
\begin{Corollary}
\label{cor:GS}
Let $(X,R)$ be a presentation satisfying the GGS condition
for some degree function $D$ on $(F(X),X)$ and $t_0\in (0,1)$.
Let $G=\Gp(X,R)$ be the pro-$p$ group defined by this presentation.
Then $G$ is infinite, and moreover, the series $Hilb_{D,\Fp[[G]]}(t_0)$ diverges.
\end{Corollary}
\vskip -.34cm

The next result shows that the Hilbert series $Hilb_{D,\Fp[[G]]}(t)$
can be expressed directly in terms of the $D$-filtration of $G$.

\begin{Proposition}
\label{GGS2}
Let $G$ be a pro-$p$ group with a finite generating set $X$,
and let $D$ be some degree function on $(F(X),X)$.
Let $\{G_n\}$ and $\{A_n\}$ be the $D$-filtrations on $G$ and $A=\Fp[[G]]$, respectively.
Let $a_n(G)=\dim_{\Fp}(A_n/A_{n+1})$ and $c_n(G)=\dim_{\Fp}(G_n/G_{n+1})$
for $n\in\dbN$. Then
\begin{equation}
\label{eq:restricted}
\sum\limits_{n=0}^{\infty} a_n(G) t^n = \prod\limits_{n=1}^{\infty}\left(\frac{1-t^{np}}{1-t^n}\right)^{c_n(G)}.
\end{equation}
\end{Proposition}
\begin{proof}
In the special case when $D$ is the standard degree function (and hence $\{G_n\}$ is the Zassenhaus filtration on $G$),
Proposition~\ref{GGS2} is an easy consequence of Quillen's theorem~\cite{Qu}
(see \cite[Chapter~12]{DDMS} for a complete proof). Even though this special
case is sufficient for the proof of Theorem~\ref{thm:main0}, we give a proof of
Proposition~\ref{GGS2} for an arbitrary degree function $D$.

Let $L=\oplus_{n\geq 1} G_n/G_{n+1}$ and $B=gr(A)=\oplus_{n\geq 0} B_n$
where $B_n=A_n/A_{n+1}$. Then $L$ is a graded $p$-Lie algebra with respect to
operations $[gG_{n+1}, hG_{m+1}]=[g,h] G_{n+m+1}$ and $(gG_{n+1})^p=g^p G_{pn+1}$,
and similarly $B$ is a graded associative algebra with $B_n^p\subseteq B_{np}$.
We can think of $G$ as a subset of $A$; it follows from the definition
of $D$-filtrations on $G$ and $A$ that $G_n\subseteq 1+A_n$ for all $n$, and
hence we have a natural map $\iota:L\to B$ such that $\iota(g G_{n+1})=(g-1)+A_{n+1}.$
It is easy to see that $\iota$ is a homomorphism of
graded $p$-Lie algebras and thus uniquely extends to a homomorphism
of graded associative algebras $\iota_*:\calU (L)\to B$ where $\calU (L)$ is the universal
$p$-envelope of $L$. First we show that $\iota_*$ is surjective.

Identify $A$ with a quotient of $\Fp\lla U\rra$ via
\eqref{xy:GSdiag}. Let $\pi: \Fp\lla U\rra \to A$
be the corresponding surjection and let $d$ be the degree function on $\Fp\lla U\rra$
$X$-compatible with $D$. 
Then $A_n=\pi((\Fp\lla U\rra)_n)$ where
$$(\Fp\lla U\rra)_n=\{ f\in \Fp\lla U\rra : d(f)\geq n\}.$$
Let $d_j=d(u_{j})=D(x_j)$. For $n\in\dbN$ let $\theta_n$ be the composite map $(\Fp\lla U\rra)_n\to A_n\to B_n$.
Since $d$ is a degree function, $B_n$ is spanned by elements of the form
$\theta_n(u_{i_1}\ldots u_{i_s})$ with $\sum_{k=1}^s d_{i_k}=n$.
But for each such element we have equality $\theta_n(u_{i_1}\ldots u_{i_s})=\prod_{k=1}^s \theta_{d_{i_k}}(u_{i_k})$
in $B$. Finally, it is easy to see that $\theta_{d_j}(u_{j})=\iota(x_j G_{d_j+1})\in\iota(L)$
for any $j\in \dbN$. Thus, $\iota_*$ is indeed surjective.

Now observe that the left-hand side of \eqref{eq:restricted} is just the series
$Hilb_B(t)$ while the right-hand side of \eqref{eq:restricted} is equal to
$Hilb_{\calU (L)}(t)$ by the Poincare-Birkhoff-Witt theorem for graded
$p$-Lie algebras. Since $\iota_*:\calU (L)\to B$ is surjective and grading-preserving,
\eqref{eq:restricted} holds if and only if $\iota_*$ is injective.

First consider the case of finite $G$. If $|G|=p^k$, then $\dim_{\Fp}B=|G|=p^k$,
$\dim_{\Fp} L=k$ and thus (again by the Poincare-Birkhoff-Witt theorem) $\dim_{\Fp}\calU(L)=p^k$.
Since $\calU(L)$ and $B$ have the same dimension, the surjective homomorphism $\iota_*$
must also be injective.

Finally, in the general case for any $N\in\dbN$ the group $G/G_N$ is finite,
and it is easy to see that $a_n(G)=a_n(G/G_N)$ and $c_n(G)=c_n(G/G_N)$ for $n<N$.
Since we already established \eqref{eq:restricted} for $G/G_N$ and $N$ is arbitrary,
it follows that \eqref{eq:restricted} holds for $G$ as well.
\end{proof}

Finally, we state a simple test for the GGS condition.

\begin{Lemma}
\label{easy}
A presentation $(X,R)$ satisfies the GGS condition 
if and only if there exists a weight function $W$ on $(F(X),X)$ such that
$1-W(X)+W(R)<0$.
\end{Lemma}
\begin{proof} Let $F=F(X)$. The forward direction is clear since if $D$ is a degree
function on $(F,X)$ and $W$ is the $(D,t)$-weight function for some $t$, then
$1-W(X)+W(R)=1-H_{D,X}(t)+H_{D,R}(t)$.

The converse follows from the fact that any weight function can be approximated
by integral weight functions. More precisely, let $W$ be any weight function on $(F,X)$.
For each $t\in (0,1)$ let $D_t$ be the unique degree function
on $(F,X)$ such that $D_t(x)=[\frac{\log W(x)}{\log t}]$ for each $x\in X$,
and let $W_t$ be the $(D_t,t)$-weight function on $(F,X)$. Then
for any $f\in F$ we have $W_t(f)\leq W(f)$ and $W_t(f)\to W(f)$ as $t\to 1$.
Since $|X|<\infty$ and $1-W(X)+W(R)<0$, we must have
$1-W_t(X)+W_t(R)<0$ for some $t\in (0,1)$.
\end{proof}

\begin{Remark}Unfortunately, we are not aware of any reference
where Theorem~\ref{GGSineq} appears as stated in this paper. Here we explain
in detail how Theorem~\ref{GGSineq} follows from Koch's
Appendix to \cite{Ko2}. We keep all the notations from the statement of 
Theorem~\ref{GGSineq}.
The numerical inequality (2.11) in \cite[p.105]{Ko2} is easily seen to be equivalent
to the inequality of powers series in Theorem~\ref{GGSineq}, but the setup
in \cite{Ko2} is somewhat restrictive as it is assumed that 
\begin{itemize}
\item[(i)] The set $S$ of defining relators of $A$ is finite;
\item[(ii)] $A=\Fp[[G]]$ for some finitely generated pro-$p$ group $G$,
and the presentation of $A$ under consideration is obtained from some finite presentation of $G$
in a standard way (as described in this section).
\end{itemize}
Fortunately, condition (ii) is not used in the proof of (2.11).
We claim that the restriction given by (i) is also inessential.
Indeed, for $n\in\dbN$ let $S_n=\{s\in S : d(s)< n\}$
(this set is finite if $S$ satisfies the hypotheses of Theorem~\ref{GGSineq}), let $I_n$ be the ideal of 
$K\lla U\rra$ generated by $S_n$ and $A_n=K\lla U\rra/I_n$. It is clear that $H_{S_n}(t)\equiv H_S(t)\mod t^n$ and
$Hilb_{A_n}(t)\equiv Hilb_A(t)\mod t^n$. Thus if for each $n$
Theorem~\ref{GGSineq} holds for the pair $(A_n, S_n)$ (which
is proved in \cite{Ko2}), it automatically holds for $(A,S)$.
\end{Remark}

\begin{Remark} Occasionally we shall use the notion of a GGS algebra,
which is defined similarly to the group case.
If $K$ is a field, a $K$-algebra $A$ will be called GGS if
$A\cong K\lla U\rra/I$ such that for some generating set $S$ of $I$,
degree function $d$ on $K\lla U\rra$ and $t_0\in (0,1)$ we have
$1-H_{d,U}(t_0)+H_{d,S}(t_0)<0$. By Theorem~\ref{GGSineq} any GGS
algebra is infinite-dimensional.
\end{Remark}

\section{Presentations for subgroups of Golod-Shafarevich groups.}

For a finitely generated pro-$p$ group $G$ we let $\Phi(G)=[G,G]G^p$
be the Frattini subgroup of $G$.
The following basic facts will be frequently used (see \cite[Chapter~1]{DDMS}):
\begin{Claim}
\label{Frattini}
The following hold:
\begin{itemize}
\item[(1)] A subset $X$ generates $G$ (topologically) if and only if
$X$ generates $G$ modulo $\Phi(G)$.
\item[(2)] If $F$ is a free pro-$p$ group of rank $d$, then
$F/\Phi(F)\cong (\dbZ/p\dbZ)^d$, and $X\subset F$ is a free generating
set if and only if $X\!\!\!\mod \Phi(F)$ is a basis for $F/\Phi(F)$.
\end{itemize}
\end{Claim}

\subsection{Some properties of weight functions}

Let $X$ be a finite set and $F=F(X)$. In this subsection we shall
prove that the restriction of a weight function on $(F,X)$
to an open subgroup $F'$ of $F$ is a weight function on $(F',X')$
for some free generating set $X'$ of $F'$ (see Corollary~\ref{main:unmarked}).
In fact, we shall give an algorithm for constructing such a set
$X'$, and we will use this algorithm later on.

\begin{Lemma}
\label{weight0} Let $X$ be a finite set, $F=F(X)$ and $W$ a weight function on $(F,X)$.
The following hold:
\begin{itemize}
\item[(a)] $W(fg)\leq \max\{W(f),W(g)\}$ for any $f,g\in F$.
\item[(b)] $W(x_i x_j^k)=\max\{W(x_i), W(x_j)\}$ if $x_i,x_j$
are distinct elements of $X$ and $p\nmid k$.
\item[(c)] $W([f,g])\leq W(f)W(g)$.
\end{itemize}
\end{Lemma}
\begin{proof} (a) and (b) follow from the identity $fg-1=(f-1)+(g-1)+(f-1)(g-1)$
(and the definition of weight functions on groups). Similarly (c) follows from the identity
$[f,g]-1=f^{-1}g^{-1}((f-1)(g-1)-(g-1)(f-1))$.
\end{proof}

\begin{Lemma}
\label{weight1}
Let $X=\{x_1,\ldots,x_d\}$ be a finite set, $F=F(X)$ and
$W$ a weight function on $(F,X)$. Let $F'$ be a subgroup of index $p$ in $F$.
Then there exists a (free) generating set $\Xgal=\{\xgal_1,\ldots,\xgal_d\}$ of $F$
such that $W(\xgal_i)=W(x_i)$ for $1\leq i\leq d$ and $F'\supset \Xgal\setminus\{\xgal_j\}$
for some $j$.
\end{Lemma}
\begin{proof} Let $I=\{i: x_i\not\in F'\}$, and choose $j\in I$
for which $W(x_j)$ is minimal possible. Define the set $\Xgal=\{\xgal_1,\ldots,\xgal_d\}$
as follows:

If $i\in\{1,\ldots,d\}\setminus I$ or $i=j$, we set $\xgal_i=x_i$. If $i\in I\setminus\{j\}$,
we let $\xgal_i$ be the (unique) element of the form $x_i x_j^k$, with $0\leq k\leq p-1$,
which lies in $F'$ (such $k$ exists since $[F:F']=p$ and $x_j\not\in F'$).

Since $\Xgal$ is obtained from $X$ by a sequence of Nielsen transformations,
it is a free generating set, and $F'$ contains $\Xgal\setminus\{\xgal_j\}$
by construction. Finally, $W(\xgal_i)=W(x_i)$ by Lemma~\ref{weight0}(b).
\end{proof}

\begin{Lemma}
\label{lem:pdec}
Let $X$ be a finite set, $F=F(X)$, fix some $x\in X$ and let $Y=X\setminus\{x\}$.
Let $F'=\la Y\ra \Phi(F)$. The following hold:
\begin{itemize}
\item[(a)] $F'$ is the unique subgroup of $F$ of index $p$ containing $Y$.
\item[(b)] $F'$ is freely generated by the set
\begin{equation}
\label{pdec:1}
X'=\cup_{y\in Y}\{y, [y,x],[y,x,x], \ldots, [y,\underbrace{\! x,\ldots, x]}_{p-1\mbox{ times }}\}\cup\{x^p\};
\end{equation}
\end{itemize}
\end{Lemma}
\begin{proof}
(a) is obvious. To prove (b) note that $\{1,x,\ldots, x^{p-1}\}$ is a transversal
for $F'$ in $F$, and the Schreier rewriting process yields that the set
$\Xgal=\cup_{y\in Y}\{y, y^x, y^{x^2}, \ldots, y^{x^{p-1}}\}\cup\{x^p\}$ generates $F'$.
It is easy to prove by induction that for any $z\in F$ and $k\in\dbN$ we have
\begin{equation*}
[z,\underbrace{x,\ldots, x]}_{k\mbox{ times }}=f_k(x,z)\cdot z^{x^k} \mbox{ where } f_k(x,z)\in\la z^{x^i}: 0\leq i<k\ra.
\end{equation*}
Thus, $\la X'\ra=\la \Xgal \ra=F'$. Finally, since $|X'|=p(|X|-1)+1$,
by Schreier formula $X'$ is a free generating set for $F'$.
\end{proof}

\begin{Lemma}
\label{weight2} Let $X=\{x_1,\ldots,x_d\}$ be a finite set, $F=F(X)$ and $W$ a weight function on $(F,X)$.
Let $F'$ be a finitely generated (pro-$p$) subgroup of $F$, let $X'=\{x'_1,\ldots,x'_e\}$ be a free generating set of $F'$,
and let $W'$ be the unique weight
function on $(F',X')$ such that $W'(x')=W(x')$ for any $x'\in X'$. The following hold:
\begin{itemize}
\item[(a)] For any $f\in F'$ we have $W'(f)\geq W(f)$.
\item[(b)] Assume that one of the following holds:
\begin{itemize}
\item[(i)] $F'=F$ (so that $d=e$) and $W(x_i)=W(x_i')$ for $1\leq i\leq d$.
\item[(ii)] $F'$ is the unique subgroup of $F$ of index $p$ containing $X\setminus\{x\}$
for some $x\in X$, and $X'$ is given by \eqref{pdec:1}.
\end{itemize}
Then $W'(f)=W(f)$ for any $f\in F'$, and thus the restriction of $W$
to $F'$ is a weight function on $(F',X')$.
\end{itemize}
\end{Lemma}
\begin{proof} The following notations will be used in all parts of the proof.
Let $U=\{u_1,\ldots, u_d\}$, and let $\iota: F\to \Fp\lla U \rra$
be the unique homomorphism such that $\iota(x_i)=1+u_i$.
By definition there is a weight function $w$ on $\Fp\lla U \rra$ such that
$$W(f)=w(\iota(f)-1)\mbox{ for any }f\in F.$$

Similarly, let $U'=\{u'_1,\ldots, u'_e\}$ and $\iota':F'\to \Fp\lla U' \rra$
the unique homomorphism such that $\iota'(x'_i)=1+u'_i$.
Then there exists a weight function $w'$ on $\Fp\lla U' \rra$ such that
$$W'(f)=w'(\iota'(f)-1)\mbox{ for any }f\in F'. \eqno(***)$$

(a) Let $\phi: \Fp\lla U' \rra\to \Fp\lla U \rra$ be the unique homomorphism
making the following diagram commutative:
\begin{equation*}
\xymatrix{
F'\ar[d]_{\iota'}\ar[rd]^{\iota_{| F'}}&\\
\Fp\lla U' \rra \ar[r]_{\phi}& \Fp\lla U \rra}
\end{equation*}
In other words, $\phi$ is defined by $\phi(u'_i)=\iota(x'_i)-1$
(such $\phi$ exists since for any $f\in F$ the element $\iota(f)-1$
lies in the ideal of $\Fp\lla U \rra$ generated by $U$).

For any $f\in F'$ we have
$W(f)=w(\iota(f)-1)=w(\phi(\iota'(f))-1)=w(\phi(\iota'(f)-1)).$
Thus, in view of (***), to prove (a) it suffices to show that
$$w'(h)\geq w(\phi(h))\mbox{ for any } h\in \Fp\lla U' \rra.$$

First note that for any $u'_i\in U'$ we have
$w'(u'_i)=W'(x'_i)=W(x'_i)=w(\phi(u'_i))$. Since $w',w$ and
$\phi$ preserve multiplication, we have
$w'(h)= w(\phi(h))$ whenever $h$ is a $U'$-monomial
(that is, a monomial in $U'$).

Now take any $h\in \Fp\lla U' \rra$, and write
$h=\sum c_{\alpha} m'_{\alpha}$ where $c_{\alpha}\in\Fp$
and $\{m'_{\alpha}\}$ are distinct monic $U'$-monomials.
Then $w'(h)=\max\{w'(m'_{\alpha}) : c_{\alpha}\neq 0\}$
by (iv) in the definition of a weight function, while (iv)'' yields
$$w(\phi(h))=w(\sum c_{\alpha}\phi(m'_{\alpha}))\leq \max\{w(\phi(m'_{\alpha})) : c_{\alpha}\neq 0\}=w'(h).$$
\vskip .1cm

(b)(i) Define the isomorphisms $\theta:F\to F$ and $\psi:\Fp\lla U\rra\to \Fp\lla U'\rra$
by $\theta(x_i)=x_i'$ and $\psi(u_i)=u_i'$ for $1\leq i\leq d$. Then we have a commutative 
diagram slightly different from the one in the proof of (a):
\begin{equation}
\label{diag:b1}
\xymatrix{
F\ar[d]_{\iota}\ar[r]^{\theta}& F\ar[d]^{\iota'}\\
\Fp\lla U \rra \ar[r]_{\psi}& \Fp\lla U' \rra}
\end{equation}
We claim that 
\begin{equation}
\label{eq:b1}
W(f)=W'(\theta(f)) \mbox{ for any } f\in F.
\end{equation}
Indeed, in view of \eqref{diag:b1}, it is enough
to prove that 
$$w(h)=w'(\psi(h)) \mbox{ for any } h\in \Fp\lla U \rra.$$
Since $w$ (resp. $w'$) is a weight function on $\Fp\lla U \rra$
(resp. $\Fp\lla U' \rra$) and $\psi:\Fp\lla U\rra\to \Fp\lla U'\rra$
is a ring isomorphism which sends $U$-monomials to $U'$-monomials,
we are reduced to showing that $w(u_i)=w'(\psi(u_i))$ for $1\leq i\leq d$.
The latter is proved by the following chain of equalities
$$w(u_i)=W(x_i)=W(x_i')=W'(x_i')=w'(u_i')=w'(\psi(u_i)),$$
where the first and fourth equalities hold by the definition of $w$ and $w'$
and the second and third ones hold by the hypotheses of Lemma~\ref{weight2}.

Now given $\delta>0$, let $F_{\delta}=\{f\in F: W(f)<\delta\}$
and $F'_{\delta}=\{f\in F: W'(f)<\delta\}$. Then $F_{\delta}$
and $F'_{\delta}$ are both open subgroups of $F$, and by Lemma~\ref{weight2}(a) 
we have $F'_{\delta}\subseteq F_{\delta}$. On the other hand,
\eqref{eq:b1} implies that $[F:F'_{\delta}]=[F: F_{\delta}]$.
Combining these two facts, we conclude that $F'_{\delta}= F_{\delta}$
for any $\delta>0$, which is equivalent to the assertion of (b)(i).

\vskip .1cm
(b)(ii) Without loss of generality we can assume that $x=x_1$.
Choose a total order on the set $U$ such that $u_1$ is the smallest element. Consider
the following order on the set
of monic $U$-monomials $M=\{u_{i_1}\ldots u_{i_k} : u_{i_j}\in U\}$:
\vskip .1cm
\centerline
{$m<\mgal$ if either $w(m)<w(\mgal)$, or $w(m)=w(\mgal)$ and $m<\mgal$ lexicographically.}
\vskip .1cm

Given nonzero $f\in \Fp\lla U\rra$, we define the leading term of $f$, denoted $\LT(f)$, to
be the largest monic $U$-monomial which appears in $f$ with nonzero coefficient.
We also set $\LT(0)=0$. Note that
\begin{itemize}
\item[(1)] $\LT(fh)=\LT(f)\LT(h)$ for any $f,h\in \Fp\lla U\rra$
\item[(2)] $w(f)=w(\LT(f))$ for any $f\in \Fp\lla U\rra$
\item[(3)] If $\{f_{\alpha}\}$ is a collection of elements of $\Fp\lla U \rra$
with distinct leading terms and $c_{\alpha}\in\Fp$, then
$\LT(\sum c_{\alpha} f_{\alpha})$ is the largest element of the
set $\{\LT(f_{\alpha}): c_{\alpha}\neq 0\}$.
\end{itemize}
From (2), (3) and the proof of (a) it is clear that to prove (b)(ii)
it suffices to show that the elements of the set
$Z=\{\phi(m'_{\alpha}): m'_{\alpha} \mbox{ is a monic $U'$-monomial}\}$
have distinct leading terms.

Any element of $Z$ can be uniquely written as a product of elements
of the set $\{\phi(u'): u'\in U'\}=\{\iota(x')-1: x'\in X'\}$. Recall that
$$X'=\cup_{y\in X\setminus\{x_1\}}\{[y,\underbrace{x_1\ldots x_1]}_{k\mbox{ times }} :
0\leq k\leq p-1\}\cup\{x_1^p\}.$$
Clearly $\LT(\iota([x_j,\underbrace{x_1,\ldots, x_1}_{k\mbox{ times }}])-1)=u_j u_1^k$
for any $j\geq 2$  and $\LT(\iota(x_1^p)-1)=u_1^p$. Property (1) above
easily implies that the elements of $Z$ have distinct leading terms.
\end{proof}

\begin{Corollary}
\label{main:unmarked}
Let $X$ be a finite set, $F=F(X)$ and $W$ a weight function on $(F,X)$.
Let $F'$ be an open subgroup of $F$. Then there exists a free generating set
$X'$ of $F'$ such that the restriction of $W$ to $F'$ is a weight function on $(F',X')$.
\end{Corollary}
\begin{proof}
It is enough to consider the case when $F'$ has index $p$ in $F$.
In this case the result follows directly from Lemma~\ref{weight2}(b)
and Lemma~\ref{weight1}.
\end{proof}

We finish this section with the converse of Lemma~\ref{weight2}(b)(ii).
While this result is not critical for our purposes, it will help clarify the exposition.

\begin{Proposition}
\label{weight6}
Let $X$ be a finite set, $F=F(X)$ and $W$ a weight function on $(F,X)$.
Let $X'$ be another free generating set for $F$ such that $W$ is a weight function on $(F,X')$.
Then there exists a bijection $\sigma: X'\to X$ such that
$W(\sigma(x))=W(x)$ for any $x\in X$. In particular, $W(X')=W(X)$.
\end{Proposition}

Before proving Proposition~\ref{weight6} we introduce an important definition
and establish an auxiliary result.

\begin{Definition}\rm
Let $X$ be a finite set, $F=F(X)$, and fix some $x\in X$.
\begin{itemize}
\item[(a)] An element $f\in F$ will be called \textit{$X$-linear in $x$} if
$f\not\in \la X\setminus\{x\}\ra \Phi(F)$. We will say \textit{`linear in $x$'}
instead of \textit{`$X$-linear in $x$'} when $X$ is clear from the context.
\item[(b)] We will say that $f\in F$ is \textit{linear} if $f\not\in\Phi(F)$.
Clearly, $f\in F$ is linear if and only if $f$ is $X$-linear in some $x\in X$.
\end{itemize}
\end{Definition}

The following straightforward claim explains this terminology:
\begin{Claim}
\label{Xlinear}
Let $X=\{x_1,\ldots, x_d\}$, $U=\{u_1,\ldots, u_d\}$,
and embed $F=F(X)$ in $\Fp\lla U\rra$ by $x_i\mapsto 1+u_i$.
Then $f\in F$ is $X$-linear in $x_k$ if and only if
the expansion of $f$ as a power series in $\{u_1,\ldots, u_d\}$ contains a term $c\, u_k$ where $c$ is a nonzero element of $\Fp$.
\end{Claim}

\begin{Lemma}
\label{weight7} Let $X=\{x_1,\ldots,x_d\}$ be a finite set, $F=F(X)$
and $W$ a weight function on $(F,X)$. Take $f\in F$, and write $f=f_L f_Q$ where
$f_L=x_1^{k_1}\ldots x_d^{k_d}$ with $0\leq k_i\leq p-1$ and $f_Q\in \Phi(F)$
(such factorization is unique).
The following hold:
\begin{itemize}
\item[(a)] $f$ is $X$-linear in $x_i$ if and only if $k_i\neq 0$.
\item[(b)] $W(f_L)=\max\{W(x_i) : f \mbox{ is $X$-linear in } x_i\}$.
\item[(c)] $W(f)=\max\{W(f_L), W(f_Q)\}$.
\item[(d)] If $f$ is $X$-linear in $x_i$, then $W(f)\geq W(x_i)$.
\end{itemize}
\end{Lemma}
\begin{proof} (a) is obvious. Let $U=\{u_1,\ldots, u_d\}$, embed $F=F(X)$ in $\Fp\lla U\rra$ by $x_i\mapsto 1+u_i$.
By definition, there exists a weight function $w$ on $\Fp\lla U\rra$ such that
$W(h)=w(h-1)$ for each $h\in F$.

Note that $f_L-1=\sum_{i=1}^d k_i u_i + r$ where each term in $r$ has degree $\geq 2$
and involves only $u_i$'s with $k_i\neq 0$. Thus,
$$W(f_L)=w(f_L-1)=\max\{w(u_i): k_i\neq 0\}=\max\{W(x_i) : k_i\neq 0\}.$$
In view of (a), this proves (b).

Since the expansion of $f_Q$ in $u_1,\ldots, u_d$ has no linear terms, the terms
of maximal $w$-weight in the expansions of $f_L-1$ and $f_Q-1$ are distinct.
Since $f-1=f_L f_Q-1=(f_L-1)+(f_Q-1)+(f_L-1)(f_Q-1)$, (c) follows.
Finally, (d) is a direct consequence of (b) and (c)
\end{proof}

\begin{proof}[of Proposition~\ref{weight6}]
Let $\Fgag=F/\Phi(F)$, and for each $f\in F$ let $\bar f$ be the image of $f$ in $\Fgag$.
Define the function $\Wgag : \Fgag \to [0,1)$ by setting
$$\Wgag(\bar f)=\inf\{W(h) : \bar h=\bar f\}$$
We claim that
$$\Wgag(\bar f)=\max\{W(x_i): f\mbox{ is $X$-linear in } x_i\} \eqno (***)$$
Indeed, define $f_L$ as in Lemma~\ref{weight7}. Then
$\overline{f_L}=\overline{f}$, whence $\Wgag(\bar f)\leq W(f_L)$.
On the other hand, for any $h\in F$ with $\bar h=\bar f$ we
have $h_L=f_L$, so by Lemma~\ref{weight7}(c) $W(h)\geq W(h_L)=W(f_L)$.
Thus, $\Wgag(\bar f)=W(f_L)$, and $W(f_L)=\max\{W(x_i): f\mbox{ is $X$-linear in } x_i\}$
by Lemma~\ref{weight7}(b).

Now given $w\in [0,1)$, let $n_w=\big|\{x\in X: W(x)\leq w\}\big|$. It follows
from (***) that $$\big|\fbar \in \Fgag:\,\, \Wgag(\bar f)\leq w\big|= p^{n_w}.$$
Thus, $n_w$ is uniquely determined by $W$ (not by $X$), and the assertion
of Proposition~\ref{weight6} easily follows.
\end{proof}

\subsection{Transformations of presentations} Recall that if $(X,R)$ is a presentation,
$\Gp(X,R)$ denotes the pro-$p$ group defined by this presentation,
that is, $\Gp(X,R)=F(X)/{\la R\ra}^{F(X)}$.

\begin{Definition}\rm Let $(X,R)$ be a presentation, $F=F(X)$ and
$\pi:F\to \Gp(X,R)$ the natural surjection. Another presentation
$(X',R')$ will be called a \textit{subpresentation} of $(X,R)$ if
\begin{itemize}
\item[(a)] $X'$ is a finite subset of $F$.
\item[(b)] The (closed) subgroup $F'=\la X'\ra$ of $F$ is freely generated by $X'$
\item[(c)] If $\pi'$ is the restriction of $\pi$ to $F'$, then
$R'$ generates $\Ker\pi'$ as a (closed) normal subgroup of $F'$.
\end{itemize}
Note that the group $\Gp(X',R')$ can be canonically identified
with the subgroup $H=\Im\pi'$ of $G=\Gp(X,R)$.
\end{Definition}

\vskip .12cm
\begin{Definition}\rm
A \textit{transformation} $T$ is an ``operation'' of replacing a presentation
$(X,R)$ by its subpresentation $(X',R')$. We shall symbolically write
$T:(X,R)\to (X',R')$.
\end{Definition}

We now describe four types of transformations which we call elementary.
\vskip .12cm

{\bf 1. $p$-descent.} Let $(X,R)$ be a presentation, $F=F(X)$, $G=\Gp(X,R)$
and $\pi:F\to G$ the natural surjection. Let $x\in X$ be such that
\begin{equation} 
\label{eq:pdnr}
\mbox{ no relator from  $R$ is $X$-linear in $x$.}
\end{equation}
Let $Y=X\setminus\{x\}$, and define $X',R'\subset F(X)$ by setting
$$X'=\cup_{y\in Y}\{y, [y,x],[y,x,x], \ldots, [y,\underbrace{\! x,\ldots, x]}_{p-1\mbox{ times }}\}\cup\{x^p\};$$
$$ R'=\{[r,\underbrace{x,\ldots, x]}_{k\mbox{ times }} : r\in R,\,\, 0\leq k\leq p-1\}.$$
We claim that $(X',R')$ is a subpresentation of $(X,R)$ and $\Gp(X',R')$ is a subgroup
of index $p$ in $G=\Gp(X,R)$. The transformation $(X,R)\to (X',R')$ will be called the
\textit{$p$-descent at $x$}.

Indeed, by Lemma~\ref{lem:pdec} $X'$ is a free generating set for the subgroup
$F'=\la Y\ra \Phi(F)$ which has index $p$ in $F$. Since $R\subset F'$ by \eqref{eq:pdnr}
(and hence $\Ker\pi\subseteq F'$), the subgroup $H=\pi(F')$ has index $p$ in $G$.
Furthermore, $\{1,x,\ldots, x^{p-1}\}$ is a transversal for $F'$ in $F$, so
the subset $R''=\{r^{x^k}: r\in R,\,\, 0\leq k\leq p-1\}$ generates
$\Ker\pi$ as a normal subgroup of $F'$. Thus, $(X',R'')$ is a subpresentation of $(X,R)$
and $\Gp(X',R'')=H$. The same argument as in the proof of Lemma~\ref{lem:pdec} shows that
$\la R''\ra=\la R'\ra$. Thus, $(X',R')$ is also a subpresentation
of $(X,R)$ with $\Gp(X',R')=\Gp(X',R'')=H$.
\vskip .1cm

The next three types of transformations do not change the group, that is,
they replace a presentation $(X,R)$ of a group $G$ by another presentation $(X',R')$ of $G$.

{\bf 2. Change of generators.} A \textit{change of generators} is a transformation
of the form $(X,R)\to (X',R)$ where $X'$ is a free generating set of $F(X)$.
Clearly, $\Gp(X,R)=\Gp(X',R)$.
A basic change of generators is given by a Nielsen transformation, that is, we set $X'=X\setminus\{x\}\cup\{xx_1^{\pm 1}\}$ where $x$ and $x_1$ are distinct elements of $X$.
\vskip .1cm

{\bf 3. Change of relators.} A \textit{change of relators} is a transformation
of the form $(X,R)\to (X,R')$ (where $R$ and $R'$ generate the same
normal subgroup of $F(X)$). By definition $\Gp(X,R)=\Gp(X,R')$.
A basic change of relators is given by setting $R'=R\setminus\{r\}\cup\{rr_1^{\pm 1}\}$ where
$r$ and $r_1$ are distinct elements of $R$.
 \vskip .1cm

{\bf 4. Cleanup.} Suppose that $(X,R)$ is a presentation
such that $X\cap R\neq \emptyset$, and let $C\subseteq X\cap R$.
We can construct a new presentation $(X',R')$ of the same group
by eliminating $C$ both from sets of generators and relators
and replacing each element of $C$ by $1$ in the remaining relators.

Formally, we set $X'=X\setminus C$ and $R'=\phi(R\setminus C)$
where $\phi: F(X)\to F(X')$ is the homomorphism which acts as identity
on $F(X')$ and sends all elements of $C$ to $1$.
It is an easy exercise to check that $\Gp(X',R')=\Gp(X,R)$.
The transformation $(X,R)\to (X',R')$ will be called \textit{the cleanup of the set $C$}.

\vskip .12cm

Finally, we introduce one more type of transformations (which will not
be considered elementary). It composes a change of generators with a cleanup
and shows the usefulness of cleanups.

{\bf Pair elimination.} Let $(X,R)$ be a presentation. Suppose that
some $r\in R$ is $X$-linear in some $x\in X$. By Claim~\ref{Frattini}(b)
$\Xgal=X\setminus\{x\}\cup\{r\}$ is a free generating set for $F(X)$.
Let $(X',R')$ be the presentation obtained from $(X,R)$ by first making
the change of generators $(X,R)\to (\Xgal,R)$ followed by the cleanup
of the singleton set $\{r\}$ (so that $X'=X\setminus\{x\}$). The transformation
$(X,R)\to (X',R')$ will be called the \textit{elimination of the pair $(x,r)$}.

\subsection {Weight-preserving transformations}
\begin{Definition}\rm A \textit{weighted presentation} is a triple $(X,R,W)$
where $(X,R)$ is a presentation and $W$ is a weight function on $(F(X),X)$.
\end{Definition}

Suppose that $(X,R,W)$ is a weighted presentation and $(X,R)\to (X',R')$
is an elementary transformation. We shall be concerned with the following questions.
\begin{itemize}
\item[(a)] Let $W'$ be the restriction of $W$ to $F(X')$. Is $W'$ a weight function on $(F(X'),X')$?
In other words, is $(X',R',W')$ a weighted presentation?
\item[(b)] What is the relationship between $W(R)$ and $W(R')$?
\end{itemize}
Affirmative answer to question (a) for most elementary transformations
was obtained in Subsection 3.1. We now analyze question (b).

\begin{Definition}\rm Let $(X,R,W)$ be a weighted presentation and $F=F(X)$.
\begin{itemize}
\item[(a)] A change of generators $(X,R)\to (X',R)$ will be called \textit{$W$-good} if
$W$ is a weight function on $(F,X')$.

\item[(b)] A change of relators $(X,R)\to (X,R')$ will be called \textit{$W$-good} if
$W(R')\leq W(R)$.

\item[(c)] A \textit{$W$-good elementary transformation} is any $p$-descent, any
cleanup or a $W$-good change of generators or relators.

\item[(d)] A transformation $(X,R)\to (X',R')$ will be called \textit{$W$-good}
if it is obtained by a sequence of $W$-good elementary transformations.
\end{itemize}
If $D$ is a degree function on $(X,R)$, a transformation $(X,R)\to (X',R')$
will be called \textit{$D$-good} if it is $W$-good where $W$ is the $(D,t)$-weight function on
$(F(X),X)$ for some $t\in (0,1)$ (clearly, the value of $t$ is not essential).
\end{Definition}

\begin{Lemma}
\label{valuations}
Let $(X,R,W)$ be a weighted presentation and $T:(X,R)\to (X',R')$
a $W$-good elementary transformation.
\begin{itemize}
\item[(a)] The restriction of $W$ to $F(X')$ is a weight function on $(F(X'),X')$.
\item[(b)] Suppose that $T$ is the $p$-descent at some $x\in X$, let $\tau=W(x)$ and $c=\frac{1-\tau^p}{1-\tau}$.
Then $$W(R')\leq c W(R)\mbox{ and }W(X')-1=c(W(X)-1).$$
\item[(c)] Suppose that $T$ is a change of generators or relators. Then $W(R')\leq W(R)$ and $W(X')=W(X)$.
\item[(d)] Suppose $T$ is the cleanup of a set $C\subseteq X\cap R$. Then $W(X')=W(X)-W(C)$ and $W(R')\leq W(R)-W(C)$.
\end{itemize}
Now assume that the transformation $T:(X,R)\to (X',R')$ is $D$-good
for some degree function $D$ on $(F(X),X)$.
\begin{itemize}
\item[(e)] If $T$ is a $p$-descent at some $x\in X$ and $n=D(x)$, then
$$\frac{1-H_{X'}(t)+H_{R'}(t)}{1-t}\leq \frac{1-H_X(t)+H_R(t)}{1-t}\cdot \frac{1-t^{pn}}{1-t^n}$$
as power series (where all Hilbert series are with respect to $D$).
\item[(f)] If $T$ is a change of generators or relators or a cleanup, then
$$\frac{1-H_{X'}(t)+H_{R'}(t)}{1-t}\leq\frac{1-H_X(t)+H_R(t)}{1-t}$$
as power series.
\item[(g)] Let $G=\Gp(X,R)$, $G'=\Gp(X',R')$, and let $\{G_n\}$ (resp. $\{G'_n\}$)
be the $D$-filtration of $G$ (resp. $G'$). Then $G'_n=G'\cap G_n$.
\end{itemize}
\end{Lemma}
\begin{proof}
(a) If $T$ is a change of generators, the assertion holds by definition.
If $T$ is a change of relators or a cleanup, the assertion is obvious.
If $T$ is a $p$-descent, the assertion holds by Lemma~\ref{weight2}(b)(ii).

\vskip .1cm
(b) For any $f,g\in F(X)$ we have $W([f,g])\leq W(f)W(g)$, and thus
$$W([r,\underbrace{x,\ldots, x]}_{k\mbox{ times }})\leq W(r)W(x)^k=W(r)\tau^k,$$
which implies that $W(R')\leq c\, W(R)$. Furthermore, computation of leading terms in the proof of
Lemma~\ref{weight2}(b)(ii) shows
that $W([y,\underbrace{x,\ldots, x]}_{k\mbox{ times }})= W(y)\tau^k$
for any $y\in X\setminus\{x\}$ and $W(x^p)=\tau^p$.
Thus, $W(X')=(W(X)-\tau)c+ \tau^p$, so
$W(X')-1=c W(X)-c\tau +\tau^p-1=c(W(X)-1)$.
\vskip .1cm
(c) If $T$ is a change of generators, the equality $W(X')=W(X)$
holds by Proposition~\ref{weight6}. The other assertions are obvious.
\vskip .1cm

(d) It is clear that $W(X')=W(X\setminus C)= W(X)-W(C)$ and $W(R\setminus C)=W(R)-W(C)$.
Recall that $R'=\phi(R\setminus C)$ where $\phi:F(X)\to F(X')$ is the homomorphism
which acts as identity on $X'$ and sends $C$ to $\{1\}$. Clearly, application of $\phi$
cannot increase the weights, so $W(R')\leq W(R\setminus C)=W(R)-W(C)$.
\vskip .1cm

(e) First note that (b) applied to the $(D,t)$-weight function
implies that (e) holds as a numerical inequality for any $t\in (0,1)$.

The proof of (b) also shows that $H_{X'}(t)-1=(H_X(t)-1)\cdot \frac{1-t^{pn}}{1-t^n}$
as power series. Since $H_{\sqcup_{i=1}^{\infty} S_{i}}(t)=\sum_{i=1}^{\infty}
H_{S_i}(t)$ as power series for pairwise disjoint sets $\{S_i\}$,
to finish the proof of (e) it suffices to check that
\begin{equation}
\label{powerseries}
\frac{H_{R'}(t)}{1-t}\leq \frac{H_R(t)}{1-t}\cdot \frac{1-t^{pn}}{1-t^n}
\mbox{ as power series when } |R|=1.
\end{equation}
If $R=\{r\}$ with $D(r)=m$, then $H_R(t)=t^m$, and the proof of (b) shows that
$H_{R'}(t)=\sum_{i=0}^{p-1}t^{m+in+\eps_i}$
with $\eps_i\in\dbZ_{\geq 0}\cup {\infty}$ (where we set $t^{\infty}=0$).
Thus, $$H_R(t)\cdot \frac{1-t^{pn}}{1-t^n}-H_{R'}(t)=
\sum_{i=0}^{p-1}t^{m+in}(1-t^{\eps_i}).$$ This implies \eqref{powerseries}
since the power series $\frac{\sum_{i=0}^{p-1}t^{m+in}(1-t^{\eps_i})}{1-t}$
clearly has non-negative coefficients.

(f) follows from the proof of (d) similarly to how (e) was deduced from the proof of (b).

(g) The assertion appears to be automatic, but it is not. Let $F=F(X)$, $F'=F(X')$
and $\pi: F\to G$ the natural surjection. Given $g\in G'$ and $f\in F$ with $\pi(f)=g$,
we need to show that there exists $f'\in F'$ with $\pi(f')=g$ and $D(f')\geq D(f)$.

If $T$ is a change of generators or relators, then $F'=F$ and there is nothing to prove.
If $T$ is a $p$-descent, the assertion is still clear since by construction $F'=\pi^{-1}(G')$.
Finally, if $T:(X,R)\to (X',R')$ is the cleanup of a set $C\subseteq X\cap R$, we put $f'=\phi(f)$
where $\phi: F\to F'$ is the homomorphism which acts as identity on $X'$ and sends $C$ to $\{1\}$.
Then $f'\in F'$, and it is clear that $\pi(f')=\pi(f)=g$ and $D(f')\geq D(f)$.
\end{proof}

\subsection{Finitary version of the GGS inequality}

In this subsection we prove a fundamental inequality relating
the Hilbert series of a presentation of a pro-$p$ group $G$
with that of a suitable presentation of a finite index subgroup of $G$.
In particular, this inequality implies that the generalized GS condition
is preserved under the passage to finite index subgroups.

\begin{Theorem}
\label{descseq}
Let $(X,R)$ be a presentation and $D$ a degree function on $(F(X),X)$.
Let $G=\Gp(X,R)$ and $\{G_n\}$ the $D$-filtration of $G$. Let $K$ be a finite
index subgroup of $G$, and for each $n\in\dbN$ define $c_n(G/K)=\log_p [KG_n : KG_{n+1}]$
(note that $c_n=0$ for sufficiently large $n$). The following hold:

\noindent
(a) There exists a $D$-good transformation
$(X,R)\to (X',R')$ such that
\begin{itemize}
\item[(i)] $K=\Gp(X',R')$

\item[(ii)] The following inequality of power series holds:
\begin{equation}
\label{keyformula}
\frac{1-H_{X'}(t)+H_{R'}(t)}{1-t}\leq \frac{1-H_{X}(t)+H_{R}(t)}{1-t}\prod_{n=1}^{\infty} \left( \frac{1-t^{pn}}{1-t^n}\right)^{c_n(G/K)}
\end{equation}
where all Hilbert series are with respect to $D$.
\end{itemize}
(b) Assume in addition that $K\subseteq G_m$ for some $m\in\dbN$. Then in (a) we can
require that $D(x)\geq m$ for any $x\in X'$ (and hence $D(f)\geq m$ for any $f\in F(X')$).
\end{Theorem}

\begin{Remark} The following observations were made by Jaikin-Zapirain:

(1) If $K$ is a normal subgroup of $G$, Proposition~\ref{GGS2} implies that
\begin{equation}
\prod_{n=1}^{\infty} \left( \frac{1-t^{pn}}{1-t^n}\right)^{c_n(G/K)}=Hilb_{\Fp[G/K]}(t),
\label{key2}
\end{equation}
where $Hilb_{\Fp[G/K]}(t)$ is the classical Hilbert series of the group algebra $\Fp[G/K]$
corresponding to the degree function $D$. In fact, there is a natural way to define $Hilb_{\Fp[G/K]}(t)$
even when $K$ is not normal, and \eqref{key2} still holds -- we shall not prove this
fact as it will not be used in the sequel.

(2) In view of \eqref{key2}, Theorem~\ref{descseq} can be considered
as a finitary version of the GGS inequality \eqref{GGSg}. In fact,
\eqref{GGSg} can be deduced from Theorem~\ref{descseq} (this is what part (b)
is useful for). Indeed, take $m\in\dbN$, and let $(X',R')$ be the presentation
satisfying the conclusion of Theorem~\ref{descseq}(a)(b) with $K=G_m$.
It is easy to see that $Hilb_{\Fp[G/G_m]}(t)\equiv Hilb_{\Fp[[G]]}(t)\mod t^m$,
and Theorem~\ref{descseq}(b) ensures that $1-H_{X'}(t)+H_{R'}(t)\equiv 1\mod t^m$.
It follows that the (coefficient-wise) inequality of power series in \eqref{GGSg}
holds at least up to degree $m-1$. Since $m$ is arbitrary, we deduce \eqref{GGSg}.
\end{Remark}
\vskip .1cm

Theorem~\ref{descseq}(a) is reduced to the following lemma
by straightforward induction on $[G:K]$.

\begin{Lemma}
\label{desclemma}
Let $(X,R), G,D,K,\{G_n\}$ be as in Theorem~\ref{descseq}. Choose a subgroup $L\subseteq G$
of index $p$ such that $K\subseteq L$, and set $L_n=L\cap G_n$ and $c_n(L/K)=\log_p [KL_n : KL_{n+1}]$.
Then there exist a $D$-good transformation $(X,R)\to (X',R')$
and $i\in\dbN$ with the following properties:
\begin{itemize}
\item[(a)] $L=\Gp(X',R')$;
\item[(b)] $\{L_n\}$ is the $D$-filtration of $L$;
\item[(c)] $c_n(L/K)=c_n(G/K)$ for $n\neq i$ and $c_i(L/K)=c_i(G/K)-1$.
\item[(d)] The following inequality of power series holds:
$$
\frac{1-H_{X'}(t)+H_{R'}(t)}{1-t}\leq \frac{1-H_{X}(t)+H_{R}(t)}{1-t}\cdot \frac{1-t^{pi}}{1-t^i}.
$$
\end{itemize}
\end{Lemma}

\begin{proof} Let $F=F(X)$, $\pi:F\to G$ the natural surjection, and
$F'=\pi^{-1}(L)$, so that $[F:F']=p$. By Lemmas~\ref{weight1} and~\ref{weight2}(b)(i)
there exists a free generating set $\Xgal$ of $F$ such that
\begin{itemize}
\item[(i)] The change of generators $(X,R)\to (\Xgal,R)$ is $D$-good.
\item[(ii)] $F'\supseteq \Xgal\setminus\{x\}$.
\end{itemize}
We claim that
\begin{itemize}
\item[(iii)] $R$ contains no elements which are $\Xgal$-linear in $x$.
\end{itemize}
Indeed, by (ii) we have $F'=\la \Xgal\setminus \{x\}\ra\Phi(F)$, so (iii)
simply asserts that $R\subset F'$. Suppose this is not the case.
Then $\la R\ra F'=F$, whence $L=\pi(F')=\pi(\la R\ra F')=G$, a contradiction.

Let $(X',R')$ be the presentation obtained form $(\Xgal,R)$
by applying the $p$-descent at $x$ (such $p$-descent is possible by condition (iii)).
Condition (ii) implies that $\Gp(X',R')=L$, so (a) holds. Note that (b) holds by Lemma~\ref{valuations}(g).

Now let $i=D(x)$. Inequality in (d) holds by condition (i)
and Lemma~\ref{valuations}(e)(f). It remains to prove (c). For a subgroup $H$ of $G$ we set $H_n=H\cap G_n$ and $c_n(H)=\dim_{\Fp}(H_n/H_{n+1})$.
Then it is easy to see that $c_n(G/K)=c_n(G)-c_n(K)$ and $c_n(L/K)=c_n(L)-c_n(K)$. Since $L$
is a subgroup of index $p$ in $G$, there exists unique $j\in\dbN$ such that $c_n(L)=c_n(G)$ for $n\neq j$
and $c_j(L)=c_j(G)-1$, and we only need to prove that $j=i$. It will suffice to show that $L_i\neq G_i$
and $L_{i+1}=G_{i+1}$.

The first assertion is clear since $\pi(x)\in G_i\setminus L$ by construction.
For the second assertion note that any $f\in F\setminus F'$ is $\Xgal$-linear in $x$
and thus $D(f)\leq D(x)=i$ by Lemma~\ref{weight7}(d). Hence
$G_{i+1}=\{\pi(f) : f\in F \mbox{ and }D(f)\geq i+1\}\subseteq \pi(F')=L$,
and so $L_{i+1}=G_{i+1}\cap L=G_{i+1}$.
\end{proof}

\begin{proof}[of Theorem~\ref{descseq}(b)]
Suppose that the presentation $(X',R')$ in Theorem~\ref{descseq}(a) does not satisfy the 
conclusion of (b), so there exists $x\in X'$ with $D(x)<m$. We can assume that $x$
has the smallest $D$-degree among all elements of $X'$.

Let $F'=F(X')$, $\pi: F'\to K$ the natural surjection and $N=\Ker\pi$.
Since $K\subseteq G_m$, there exists $f\in F'$ such that $D(f)\geq m$ and
$\pi(f)=\pi(x)$ (here we use Lemma~\ref{valuations}(g)). Note
that $f$ is not $X'$-linear in $x$ (for otherwise $D(f)= D(x)<m$). Thus,
$f^{-1}x$ is an element of $N$ which is $X'$-linear in $x$, whence some $r\in R'$
must be $X'$-linear in $x$ (as $N=\la R'\ra^{F'}$). Note that $D(r)=D(x)$ by the choice of $x$.

Let $(X'',R'')$ be the presentation obtained from $(X',R')$ by eliminating the pair $(x,r)$.
Since $D(r)=D(x)$, this pair elimination is $D$-good by Lemma~\ref{weight2}(b)(i),
so by Lemma~\ref{valuations}(f) inequality~\eqref{keyformula} holds with $(X',R')$ replaced by $(X'',R'')$.

If the presentation $(X'',R'')$ does not satisfy the conclusion of
Theorem~\ref{descseq}(b), we repeat the above procedure. Since each application
of this procedure decreases the size of the generating set by $1$, after finitely many
steps we obtain a presentation with desired property.
\end{proof}

\subsection{GGS condition for subgroups of GGS pro-$p$ groups}

\def\Pres{{\mathbf{Pres}}}

\begin{Definition}\rm Let $w,\delta$ and $\eps$ be positive real numbers. We will say
that a weighted presentation $(X,R,W)$ \textit{satisfies the condition $\Pres(w,\delta,\eps)$} if
\begin{itemize}
\item[(a)] $W(X)=w$ and $W(R)<\delta$;
\item[(b)] $\sum_{w \in \Im(W)} w <\eps$ where
$\Im(W)=\{W(f) : f\in F(X)\}$ is the image of $W$.
\end{itemize}
\end{Definition}

Our ultimate goal for applications in the next section is to
prove the following result.

\begin{Theorem}
\label{deepdescent} Let $G$ be a GGS pro-$p$ group, and let $w,\delta$ and $\eps$
be any positive real numbers. Then there exists an open normal subgroup $H$ of $G$
and a weighted presentation $(X,R,W)$ of $H$ which satisfies $\Pres(w,\delta,\eps)$.
\end{Theorem}

However, the next theorem should probably be considered the key
result of this section.

\begin{Theorem}
\label{deepdescent3} Let $G$ be a GGS pro-$p$ group. Then for any $M>0$
there exists an open normal subgroup $K$ of $G$ and a weighted presentation $(X,R,W)$
of $K$  such that
$$1-W(X)+W(R)<-M.$$
\end{Theorem}

\begin{proof}[of Theorem~\ref{deepdescent3}]
We are given that $G$ has a presentation $(X,R)$ such that
$1-H_{X}(t_0)+H_{R}(t_0)<0$ for some
$0<t_0<1$ and degree function $D$ on $(F(X),X)$ (where all Hilbert series
are with respect to $D$). Let $W$ be the $(D,t_0)$-weight function on $(F(X),X)$,
that is, $W(f)=t_0^{D(f)}$.

Let $\{G_n\}$ be the $D$-filtration of $G$. For each $n\in\dbN$ let $(X_n,R_n)$ be a presentation
satisfying the conclusion of Theorem~\ref{descseq} with $K=G_{n+1}$.
Let
\begin{align*}
&\mu=-(1-W(X)+W(R))= -(1-H_{X}(t_0)+H_{R}(t_0))&\\
&\mu_n= -(1-W(X_n)+W(R_n))=-(1-H_{X_n}(t_0)+H_{R_n}(t_0)) \mbox{ for }n\in\dbN.&
\end{align*}
Let $c_k=\dim_{\Fp}(G_k/G_{k+1})$, and define $c_k(G/G_{n+1})$ as in Theorem~\ref{descseq}.
Then clearly
$$c_k(G/G_{n+1})=\left\{
\begin{array}{ll}
c_k & \mbox{ if } k\leq n\\
0 & \mbox{ if } k> n
\end{array}
\right.$$
Thus, Theorem~\ref{descseq} yields $$\mu_n\geq \mu \prod_{k=1}^{n} \left(\frac{1-t_0^{pk}}{1-t_0^k}\right)^{c_k}.$$
Proposition~\ref{GGS2} and Corollary~\ref{cor:GS} imply that the infinite product
$\prod\limits_{k=1}^{\infty} \left(\frac{1-t_0^{pk}}{1-t_0^k}\right)^{c_k}$ diverges.
Since $\mu>0$, we have $\mu_n\to\infty$. The proof is complete.
\end{proof}

Our next result combined with Theorem~\ref{deepdescent3} implies Theorem~\ref{deepdescent}.

\begin{Theorem}
\label{deepdescent4}
Let $w,\delta,\eps$ and $M$ be positive real numbers, with $\delta,\eps<1$.
Let $K$ be a pro-$p$ group which has a weighted presentation $(X,R,W)$ such that $1-W(X)+W(R)<-M$.
\begin{itemize}
\item[(a)] Assume that $M>\max\{\frac{w^2}{\delta}, \frac{w}{\eps}\}$.
Then $K$ has a weighted presentation $(X',R',W')$ such that
$W'(X')=w$, $W'(R')<\delta$ and $\max\{W'(x): x\in X'\}<\eps$.

\item[(b)] Assume that $M> 4\max\{\frac{w^2}{\delta}, \frac{w}{\eps}\}$.
Then $K$ has a weighted presentation satisfying $\Pres(w,\delta,\eps)$.
\end{itemize}
\end{Theorem}

The desired weighted presentation of $K$ in Theorem~\ref{deepdescent4}(b) will be obtained
from the original weighted presentation $(X,R,W)$ using transformations
and contractions of weight functions as defined below.

\begin{Definition}\rm Let $X$ be a finite set, $F=F(X)$
and $W$ a weight function on $(F,X)$. Let $c\geq 1$ be a real number.
The \textit{$c$-contraction} of $W$ is the unique weight function $W'$ on $(F,X)$
such that $W'(x)=\frac{W(x)}{c}$ for any $x\in X$.
\end{Definition}

\begin{Lemma}
\label{contraction}
Let $X=\{x_1,\ldots,x_n\}$ be a finite set, $F=F(X)$, $W$ a weight function on $(F,X)$
and $W'$ the $c$-contraction of $W$ for some $c\geq 1$. The following hold:
\begin{itemize}
\item[(a)] $W'(f)\leq \frac{W(f)}{c}$ for any $f\in F$;
\item[(b)] If $f$ is not linear, that is, $f\in \Phi(F)$,
then $W'(f)\leq \frac{W(f)}{c^2}$.
\end{itemize}
Now take $f\in F\setminus \Phi(F)$, and write $f=f_L f_Q$ where
$f_L=x_1^{k_1}\ldots x_n^{k_n}$ with $0\leq k_i\leq p-1$ and $f_Q\in \Phi(F)$.
Let $\lam=W(f)/W(f_L)$. We have
\begin{itemize}
\item[(c)] $W'(f_L)=\frac{W(f_L)}{c}$;
\item[(d)] If $c\leq \lam$, then $W'(f)\leq \frac{W(f)}{c^2}$;
\item[(e)] If $c\geq \lam$, then $W'(f)=W'(f_L)$.
\end{itemize}
\end{Lemma}
\begin{proof} (a) is obvious, (b) is clear in view of Claim~\ref{Xlinear},
(c) follows directly from Lemma~\ref{weight7}(b), and (d) and (e) follow
easily from (a),(b),(c) and Lemma~\ref{weight7}(c).
\end{proof}

\begin{proof}[of Theorem~\ref{deepdescent4}]
(a) We start with a special case when $R$
contains no linear relators (we warn the reader that this
special case occurs very rarely in practice). In this case we set
$(X',R')=(X,R)$ and define the new weight function $W'$ to be the $c$-contraction
of $W$ where $c=W(X)/w$.

Clearly, $W'(X)=W(X)/c$. Since $1-W(X)+W(R)<-M$, we have $W(X)>M$
and $W(X)>W(R)$.
Thus, Lemma~\ref{contraction}(b) yields
$$W'(R)\leq\frac{W(R)}{c^2}=\frac{W(R)\cdot w^2}{W(X)^2}< \frac{w^2}{W(X)}<
\frac{w^2}{M}<\delta.$$
Since $c>\frac{M}{w}$, for any $x\in X$ we have
$$W'(x)=\frac{W(x)}{c}<\frac{1}{c}<\frac{w}{M}<\eps.$$

\textit{General case:} We need to find a way to treat linear relators efficiently.
Here is the idea. Let us say that a subset $C$ of $R$ is $W$-regular
if there exists a free generating set $\Xgal$ of $F=F(X)$
such that $C\subseteq \Xgal$ and $W$ is a weight function on $(F,\Xgal)$
(note that $W$-regular subsets may only contain linear relators,
but for a linear relator $r$ it may happen that the singleton set
$\{r\}$ is already not $W$-regular). If $C$ is a $W$-regular
set of relators, we can remove
it using a $W$-good transformation
by first making the $W$-good change of generators
$(X,R)\to (\Xgal,R)$ followed by the cleanup of $C$.
To make use of this observation, we proceed as follows.

First, we ``optimize'' the set of relators in a suitable way
(using a change of relators). Then we apply the $c$-contraction
to $W$ for suitable $c$ and show that the set $R_{bad}$ of all linear relators in $R$
can be divided into two disjoints subsets $R_{bad}^1$ and $R_{bad}^2$
with the following properties: the set $R_{bad}^1$ is regular for the contracted 
weight function, while the weights of elements of
$R_{bad}^2$ contract at least by the factor of $c^2$ (so they behave
as non-linear relators during the contraction). Then we remove all relators
from $R_{bad}^1$ as described above, and after an additional
contraction we obtain a weighted presentation with desired properties.

We now begin the formal proof.

{\bf Step 1:} \textit{Optimizing the set of relators.}

First note that $R$ contains only finitely many linear relators
since $W(R)<\infty$ while for any linear relator $r\in R$ we have
$W(r)\geq \min\{W(x): x\in X\}$.

Let $R_{bad}=\{r_1,\ldots, r_b \}$ be the set of all linear relators in $R$.
Without loss of generality, we can assume that
$W(r_1)\leq W(r_i)$ for any $i\geq 2$. Let $x_1\in X$
be the generator such that $r_1$ is linear in $x_1$
and $W(x_1)$ is largest among all generators with this property.
After applying a change of relators of the form
$r_i\to r_i r_1^{k_i}$ for $2\leq i\leq b$,
we can assume that all relators besides $r_1$ are NOT linear in $x_1$.
By the choice of $r_1$, this change of relators is
$W$-good, so the total weight of relators cannot increase.
Let $\Rgal_{bad}=\{\rgal_2,\ldots, \rgal_{\bgal} \}$ be the new set of linear relators excluding $r_1$
(note that $\bgal$ may be smaller than $b$ since some of the relators
$r_i r_1^{k_i}$ may not be linear).

Now apply the same procedure to the set $\Rgal_{bad}$:
after reordering, assume that $W(\rgal_2)\leq W(\rgal_i)$ for $i\geq 3$,
choose $x_2\in X$ such that $\rgal_2$ is linear in $x_2$
and $W(x_2)$ is largest possible etc..
Note that $x_2$ is different from $x_1$.

\vskip .12cm
After repeating this procedure at most $b$ times and changing the notations,
we can assume that the set $R_{bad}=\{r_1,\ldots,r_m\}$ of linear relators in $R$
has the following property: there exist distinct elements $x_1,\ldots, x_m\in X$
such that for $1\leq i\leq m$ we have
\begin{itemize}
\item[(i)] $r_i$ is linear in $x_i$ and not linear in $x_j$ for $j<i$;
\item[(ii)] if $r_i$ is linear in $x$ for some other $x\in X$, then $W(x)\leq W(x_i)$.
\end{itemize}
Note that (i),(ii) and Lemma~\ref{weight7}(b) imply that
\begin{itemize}
\item[(iii)] $W((r_i)_L)=W(x_i)$ for $1\leq i\leq m$ (using the notations of Lemmas~\ref{weight7} and \ref{contraction})
\end{itemize}

\vskip .12cm
{\bf Step 2:} \textit{Initial contraction.}

For $1\leq i\leq m$, let $\lam_i=\frac{W(r_i)}{W(x_i)}=\frac{W(r_i)}{W((r_i)_L)}$.
Note that $\lam_i\geq 1$ by Lemma~\ref{weight7}(c).
Define $X_{bad}=\{x_1,\ldots, x_m\}$, $X_{good}=X\setminus X_{bad}$
and $R_{good}=R\setminus R_{bad}$. We have
$$1-W(X_{bad})-W(X_{good})+W(R_{bad})+W(R_{good})<-M.$$
Since $W(r_i)\geq W(x_i)$, we have $W(R_{bad})\geq W(X_{bad})$,
whence $W(X_{good})>M$.

Let $c=W(X_{good})/w$. Note that $c>\frac{M}{w}$.

Define $W_{a}$ to be the $c$-contraction of $W$ on $(F(X),X)$. Note that
$$W_a(X_{good})=\frac{W(X_{good})}{c}=w\quad \mbox{ and }\quad W_a(X)=w+\frac{W(X_{bad})}{c}\geq w.$$
Now divide $R_{bad}$ into two disjoint subsets:
$$R_{bad}^1=\{r_i\in R_{bad} : \lam_i\leq c\}\mbox{ and }R_{bad}^2=R_{bad}\setminus R_{bad}^1,$$
and let $X_{bad}^j$ be the subset of $X_{bad}$ corresponding to $R_{bad}^j$ for $j=1,2$,
that is, $X_{bad}^j=\{x_i\in X_{bad} : r_i\in R_{bad}^j\}$.
\begin{Claim}
\label{Rbad} The following hold:
\begin{itemize}
\item[(1)] If $r\in R_{good}\sqcup R_{bad}^2$, then $W_a(r)\leq W(r)/c^2$.
\item[(2)] If $r_i\in R_{bad}^1$, then $W_a(r_i)=W_a(x_i)$.
\end{itemize}
\end{Claim}
\begin{proof} (1) holds by Lemma~\ref{contraction}(b)(d), and (2) follows
directly from  condition (iii) above and Lemma~\ref{contraction}(c)(e).
\end{proof}

\vskip .12cm
{\bf Step 3:} \textit{Cleanup and the second contraction.}

Now let $\Xgal= X\setminus X_{bad}^1\cup R_{bad}^1$.
We claim that $\Xgal$ is a free generating set of $F=F(X)$. Indeed,
condition (i) ensures that the matrix $A\in GL_m(\Fp)$ expressing
$\Xgal$ in terms of $X$ mod $\Phi(F)$ is upper-triangular
with nonzero diagonal entries, hence invertible. Thus,
$\Xgal$ generates $F$ modulo $\Phi(F)$ and hence generates $F$.
Finally, $\Xgal$ is a free generating set of $F$ since $|\Xgal|=|X|$.

Since $W_a(r_i)=W_a(x_i)$ for every $r_i\in R_{bad}^1$, the
change of generators $(X,R)\to (\Xgal,R)$ is $W_a$-good by Lemma~\ref{weight2}(b)(i).
We now apply this change of generators followed
by the cleanup of the set $R_{bad}^1$, and let $(X',R')$
be the obtained presentation (note that $X'=X\setminus X_{bad}^1$).

By Lemma~\ref{valuations}(d) we have $W_a(X')=W_a(X)-W_a(X_{bad}^1)$ and
$W_a(R')\leq W_a(R)-W_a(R_{bad}^1)$.
We shall now estimate the quotient $W_a(R')/W_a(X')$.
Using Claim~\ref{Rbad}, we have
\begin{multline*}
\frac{W_a(R')}{W_a(X')}\leq \frac{W_a(R_{good})+W_a(R_{bad}^2)}{W_a(X_{good})+W_a(X_{bad}^2)}\leq
\frac{1}{c}\cdot\frac{W(R_{good})+W(R_{bad}^2)}{W(X_{good})+W(X_{bad}^2)}=\\
\frac{1}{c}\cdot\frac{W(R)-W(R_{bad}^1)}{W(X)-W(X_{bad}^1)}\leq
\frac{1}{c}\cdot\frac{W(R)}{W(X)}<\frac{1}{c},
\end{multline*}
where the last two steps hold since $W(R)<W(X)$ and $W(R_{bad}^1)\geq W(X_{bad}^1)$.
\vskip .1cm
Now let $c_1=W_a(X')/w$.  Note that $c_1\geq 1$ since $W_a(X')\geq W_a(X_{good})=w$.
Let $W'$ be the weight function on $(F(X'),X')$ obtained by the
$c_1$-contraction of $W$. Then we have $W'(X')=W_a(X')/{c_1}=w$ and
$W'(R')\leq W_a(R')/{c_1}$, so
$$W'(R')\leq W'(X')\cdot \frac{W_a(R')}{W_a(X')}<\frac{w}{c}<\frac{w^2}{M}<\delta.$$
Finally, for any $x\in X'$ we have $W'(x)\leq W_a(x)\leq \frac{W(x)}{c}<\frac{1}{c}\leq \frac{w}{M}<\eps$.
\vskip .2cm
(b) By (a) we can find a weighted presentation $(X',R',W')$ of $K$ such that
$W'(X')=2w$, $W'(R')<\delta$ and $\max\{W'(x): x\in X'\}<\frac{\eps}{2}$.

Let $D$ be the unique degree function on $(F(X'),X')$
such that $\frac{W'(x)}{2}< (\frac{1}{2})^{D(x)}\leq W'(x)$ for any $x\in X'$.
Let $W''$ be the $(D,\frac{1}{2})$-weight function on $(F(X'),X')$, so that
$\frac{W'(x)}{2}< W''(x)\leq W'(x)$ for any $x\in X'$.

By construction $w=\frac{W'(X')}{2} <W''(X')$ and $W''(R')\leq W'(R')<\delta$.
Furthermore, if $N=\min\{D(x): x\in X'\}$, then $\sum_{z\in \Im(W'')}W''(z)\leq \sum_{n=N}^{\infty} (\frac{1}{2})^n=(\frac{1}{2})^{N-1}$ and
$$\left(\frac{1}{2}\right)^{N}=\max\{W''(x): x\in X'\}\leq \max\{W'(x): x\in X'\}<\frac{\eps}{2},$$
so $\sum_{z\in \Im(W'')}W''(z)<\eps$.

Finally, let $c'=\frac{W''(X')}{w}$ and $\Wgal$ be the $c'$-contraction of $W''$.
It is clear that the weighted presentation $(X',R', \Wgal)$ satisfies $\Pres(w,\delta,\eps)$.
\end{proof}
\begin{Remark}
In the final step of the proof of Theorem~\ref{deepdescent4}(b)
we can also take $\Wgal$ to be the $(D,\tau)$-weight function on
$(F(X'),X')$, where $\tau\in (0,\frac{1}{2})$ is such that $\Wgal(X')=w$ (such $\tau$ exists by continuity).
Thus, $\Wgal$ can be chosen an integral weight function.
\end{Remark}

\subsection{From pro-$p$ to abstract groups.}
We finish with the analogue of Theorem~\ref{deepdescent}
dealing with abstract groups (see Proposition~\ref{deepdescent2} below).
The proofs of the following two results were suggested by Jaikin-Zapirain.

\begin{Lemma}
\label{jaikin3}
Let $(X,R)$ be a presentation, $G=\Gp(X,R)$ and $\Gamma$ and $\Lambda$ two dense countable
subgroups of $G$. Let $W$ be a weight function on $(F(X),X)$. Then for any $\delta_1>0$
there exists a subset $R_{aux}$ of $F(X)$ such that $W(R_{aux})<\delta_1$ and the images
of $\Gamma$ and $\Lambda$ in $\Gp(X,R\cup R_{aux})$ coincide.
\end{Lemma}
\begin{proof}
Let $\pi:F(X)\to G$ be the natural surjection. Choose countable subsets
$Y=\{y_1,y_2,\ldots\}$ and $Z=\{z_1,z_2,\ldots\}$ of $F(X)$ such that
$\pi(Y)=\Gamma$ and $\pi(Z)=\Lambda$.

Since $\Lambda$ is dense in $G$ and $\pi$ maps open sets to open sets,
we conclude that $\pi^{-1}(\Lambda)$ is dense in $F(X)$. Thus,
for each $i\in\dbN$ we can find $y'_i\in F(X)$ such that
$\pi(y'_i)\in \Lambda$ and $W(y_i^{-1} y'_i)<\frac{\delta_1}{2^{i+1}}$. Similarly,
for each $i\in\dbN$ we can find $z'_i\in F(X)$ such that
$\pi(z'_i)\in \Gamma$ and $W(z_i^{-1} z'_i)<\frac{\delta_1}{2^{i+1}}$.

Let $R_{aux}=\{y_i^{-1} y'_i,\,\, z_i^{-1} z'_i : i\in \dbN\}$;
clearly $W(R_{aux})<\delta_1$. If $\Gamma'$ (resp. $\Lambda'$)
is the image of $\Gamma$ (resp. $\Lambda$) in $\Gp(X,R\cup R_{aux})$,
then by construction $\Gamma'\subseteq \Lambda'$ and $\Lambda'\subseteq \Gamma'$,
so $\Gamma'=\Lambda'$.
\end{proof}

{\bf Notation:} If $(X,R)$ is a presentation,
by $\Gp_{abs}(X,R)$ we will denote the abstract group generated by (the image of) $X$ in $\Gp(X,R)$.

\begin{Remark} If $R$ happens to be a subset of $F_{abs}(X)$ (the free abstract group on $X$),
the abstract group $\la X | R\ra$ need not coincide with $\Gp_{abs}(X,R)$; what is
true is that $\Gp_{abs}(X,R)$ is isomorphic to the image of $\la X | R\ra$
in its pro-$p$ completion.
\end{Remark}

\begin{Proposition}
\label{deepdescent2}
Let $G$ be a GGS pro-$p$ group, and let $w,\delta$ and $\eps$
be positive real numbers. Let $\Gamma$ be a dense finitely generated
subgroup of $G$. Then there exists a weighted presentation $(X,R,W)$ such that
\begin{itemize}
\item[(i)] $(X,R,W)$ satisfies the condition $\Pres(w,\delta,\eps)$.
\item[(ii)] The abstract group $\Gp_{abs}(X,R)$ is a quotient of a finite index subgroup of $\Gamma$.
\end{itemize}
\end{Proposition}
\begin{proof} By Theorem~\ref{deepdescent} we can find an open subgroup $\Hgal$ of $G$,
a presentation $(X,\Rgal)$ of $\Hgal$ and a weight function $W$ on $(F(X),X)$
such that $(X,\Rgal,W)$ satisfies $\Pres(w,\delta,\eps)$. Note that
$\Hgal\cap \Gamma$ and $\Gp_{abs}(X,\Rgal)$ are both dense countable subgroups
of $\Gp(X,\Rgal)$.

By Lemma~\ref{jaikin3}, given $\delta_1>0$ we can find $R_{aux}\subset F(X)$
with $W(R_{aux})<\delta_1$ such that if $R=\Rgal\cup R_{aux}$, then
the images of $\Hgal\cap \Gamma$ and $\Gp_{abs}(X,\Rgal)$ in $\Gp(X,R)$ coincide.
By definition the image of $\Gp_{abs}(X,\Rgal)$ in $\Gp(X,R)$
is $\Gp_{abs}(X,R)$, so (ii) holds.
Since $W(R)<W(\Rgal)+\delta_1$, the weighted presentation $(X,R,W)$ satisfies $\Pres(w,\delta,\eps)$
for sufficiently small $\delta_1$, so (i) also holds.
\end{proof}

\begin{Remark} The reader might wonder how the class
\vskip .1cm
\centerline{$\mathfrak{GGS}=\{$GGS abstract groups$\}$ is related to the class}
\vskip .1cm
\vskip .1cm
\centerline{$\mathfrak{GGS}'=\{$dense finitely generated subgroups of GGS pro-$p$ groups$\}$}
\vskip .1cm
\noindent
to which Proposition~\ref{deepdescent2} applies. Clearly, any group in $\mathfrak{GGS}'$ is residually-$p$,
which is not necessarily true for groups in $\mathfrak{GGS}$; however, for any $\Gamma\in\mathfrak{GGS}$
the image of $\Gamma$ in its pro-$p$ completion belongs to $\mathfrak{GGS}'$. Thus,
to prove Theorem~\ref{thm:main0} it will suffice to prove the corresponding result
for groups in $\mathfrak{GGS}'$, and in fact this is what we will do (see Theorem~\ref{jaikin5}).
On the other hand, we do not know whether any group in $\mathfrak{GGS}'$ also belongs to $\mathfrak{GGS}$.
\end{Remark}

\section{Constructing Kazhdan quotients}
In this section we prove the main result of this paper (Theorem~\ref{thm:main0}).
The first three subsections will be devoted to the proof of Theorem~\ref{thm:main0}
``up to finite index'':

\begin{Theorem}
\label{thm:main1}
Let $\Gamma$ be a Golod-Shafarevich group.
Then $\Gamma$ has a subgroup of finite index which surjects onto an infinite
group with property $(T)$.
\end{Theorem}

In the last subsection we will deduce Theorem~\ref{thm:main0} from Theorem~\ref{thm:main1}
using a very general argument.

\subsection{Golod-Shafarevich groups with property $(T)$}
In \cite{EJ}, a large class of groups, called \textit{Kac-Moody-Steinberg groups,}
was introduced, and many of these groups were shown to have property $(T)$.
We shall not recall the general construction, but concentrate on its special
case that will be used in this paper.

Let $\dbF$ be a finite field, let
$n_1,\ldots, n_k$ be a collection of positive integers, and
for each $1\leq i\leq k$ let $M_i=\dbF^{n_i}$ (the standard ${n_i}$-dimensional $\dbF$-vector space).
For $1\leq i\leq k$ let $\calU_{i}$ be the group whose elements are formal symbols
$\{x_{i}(a): a\in M_i\}$ subject to relations $$x_{i}(a)\cdot x_{i}(b)=
x_{i}(a+b),$$ so that $\calU_{i}\cong (M_i,+)$.
Define $\KMS(\dbF;n_1,\ldots, n_k)$ to be the group generated by $\{\calU_1,\ldots, \calU_k\}$
subject to the following relations:
\begin{align*}
&\mbox{ (R1) }[x_{i}(a),x_j(b),x_i(c)]=1 \mbox{ for } 1\leq i\neq j\leq k,\,\, a,c \in M_i\mbox{ and } b \in M_j.&\\
&\mbox{ (R2) }[x_{i}(\lam a),x_j(b)]=[x_{i}(a),x_j(\lam b)]\mbox{ for } 1\leq i\neq j\leq k,\,\,a\in M_i,\,\,b\in M_j\mbox{ and }
\lam\in \dbF.&
\end{align*}
In other words, $\KMS(\dbF;n_1,\ldots, n_k)$ is the largest quotient of the free product
$\calU_1*\ldots*\calU_k$
in which each pair $\{\calU_i,\calU_j\}$ generates a nilpotent subgroup of class two and the commutator map
$\calU_i\times \calU_j\to [\calU_i,\calU_j]$ is $\dbF$-linear. Note that relations (R2)
follow from (R1) if $|\dbF|$ is prime.

The group $\KMS(\dbF;n_1,\ldots, n_k)$ will be called a \textit{Kac-Moody-Steinberg group}
and its subgroups $\calU_1,\ldots, \calU_k$ will be called \textit{the root subgroups}
(for the explanation of this terminology see \cite[Section 7]{EJ}).

The following result is a special case of \cite[Corollary~7.2]{EJ}:
\begin{Theorem}
\label{GS_supply}
The group $\KMS(\dbF;n_1,\ldots, n_k)$ has property $(T)$ whenever $|\dbF|>(k-1)^2.$
\end{Theorem}
It is easy to see that ``most'' of the groups $\KMS(\dbF;n_1,\ldots, n_k)$ are Golod-Shafarevich
if $k\geq 9$ and $|\dbF|$ is prime. Thus, for $p\geq 67$, Theorem~\ref{GS_supply} produces
a large supply of Golod-Shafarevich groups with property $(T)$ which provides a starting point
for the proof of Theorem~\ref{thm:main1}. However, if $|\dbF|$ is not prime,
the group $\KMS(\dbF;n_1,\ldots, n_k)$ will never be Golod-Shafarevich because of relations (R2).
In this case, instead of working with the group $\Gamma=\KMS(\dbF;n_1,\ldots, n_k)$ we will
consider the completed group algebra $\Fp[[\Gamma_{\phat}]]$ and show that $\Fp[[\Gamma_{\phat}]]$
maps onto a finite codimension subalgebra of an $\dbF$-algebra satisfying the GGS condition.

In the next two subsections we shall prove the following theorem.

\begin{Theorem}
\label{jaikin4}
There exist positive real numbers $w,\delta$ and $\eps$
with the following property: if $(X,R,W)$ is a triple satisfying the condition
$\Pres(w,\delta,\eps)$, then the group $\Gp_{abs}(X,R)$ has
an infinite quotient with property $(T)$.
In fact, one can take $w=\frac{3}{2}$, $\delta=\frac{1}{50}$ and $\eps=\frac{1}{100}$
when $p\geq 67$, and $w=12$, $\delta=\frac{1}{100}$ and $\eps=\frac{1}{1000}$
for arbitrary $p$.
\end{Theorem}

It is clear that Theorem~\ref{thm:main1} follows from Theorem~\ref{jaikin4} and
Proposition~\ref{deepdescent2}.
\vskip .2cm

\begin{Remark} Suppose $K$ is a pro-$p$ group with a weighted presentation
$(X',R',W')$ such that $1-W'(X')+W'(R')<-6\cdot 10^4$. Then by Theorem~\ref{deepdescent4}(b)
$K$ also has a presentation $(X,R,W)$ satisfying $\Pres(w,\delta,\eps)$ where $w,\delta$ and $\eps$
are as in Theorem~\ref{jaikin4}. This observation justifies Step~2
of the algorithm presented in the introduction.
\end{Remark}

\subsection{Proof of Theorem~\ref{jaikin4} for $\mathbf{p\geq 67}$}
Fix arbitrary positive real numbers $w$, $\delta$ and $\eps$,
and let $(X,R,W)$ be a triple satisfying $\Pres(w,\delta,\eps)$.
In fact, in our case $p\geq 67$ we will not use the full power
of the assumption (b) in $\Pres(w,\delta,\eps)$ -- it will be enough
to assume that $\max\{W(x) : x\in X\}<\eps$.

Divide $X$ into $9$ (disjoint) subsets $X^1,\ldots, X^9$ such that $|W(X^i)-W(X^j)|<\eps$ for each $i,j$.
For each $1\leq i\leq 9$ we will denote the elements of $X^i$
by $x_{i,1},\ldots,x_{i,{n_i}}$.
We can now identify $X$ with a generating set for the group
$\KMS(\dbF_p;n_1,\ldots, n_9)$, which is given by the abstract presentation
\begin{multline*}
\la X\mid R_{KM}\ra=\la x_{i,k}, 1\leq i\leq 9, 1\leq k\leq n_i\mid x_{i,k}^p,\,\,
[x_{i,k},x_{i,l}], \quad [x_{i,k},x_{j,l},x_{i,m}]
\ra.
\end{multline*}
Since $p>(9-1)^2$, the group $\KMS(\dbF_p;n_1,\ldots, n_9)=\la X | R_{KM}\ra$ 
has property $(T)$ by Theorem~\ref{GS_supply}. 
Hence the group $\Gp_{abs}(X,R_{KM})$  also has $(T)$ being a quotient of 
$\la X | R_{KM}\ra$.\footnote{Recall that by definition $\Gp_{abs}(X,R_{KM})$
is the image of $\la X | R_{KM}\ra$ in its pro-$p$ completion.
We do not know whether the group $\la X | R_{KM}\ra$ is always residually-$p$
(which is equivalent to the injectivity of the map $\la X | R_{KM}\ra\to \Gp_{abs}(X,R_{KM})$).}
\vskip .1cm
Now consider the presentation $(X, R\cup R_{KM})$.
For $1\leq i\leq 9$ let $w_i=W(X^i)$, so that $\sum_{i=1}^9 w_i=W(X)=w$.
Note that
\begin{multline*}
W(\{[x_{i,k},x_{i,l}] : 1\leq i\leq 9,\, 1\leq k<l\leq n_i\})=\\
\sum_{i=1}^9 \sum_{1\leq k<l\leq n_i} W(x_{i,k}) W(x_{i,l})<
\frac{1}{2}\sum_{i=1}^9 \Big(\sum_{k=1}^{n_i} W(x_{i,k})\Big)^2=
\frac{1}{2} \sum_{i=1}^9 w_i^2.
\end{multline*}
Similarly, $$W(\{[x_{i,k},x_{j,l},x_{i,m}] :
1\leq i\neq j\leq 9,\, 1\leq k,m\leq n_i,\, 1\leq l\leq n_j\})\leq \sum_{1\leq i\neq j\leq 9} w_i^2 w_j.$$
Therefore,
$$W(R_{KM})< \sum_{i=1}^9 \frac{w_i^2}{2} + \sum_{1\leq i\neq j\leq 9}(w_i^2 w_j) + \eps^{p-1}w.$$
(since $\max\{W(x) : x\in X<\eps\}$).
Recall that we have defined subsets $X^1,\ldots, X^9$ so that $|w_i-w_j|<\eps$
for each $i,j$, whence $|w_i-\frac{w}{9}|<\eps$ for each $i$.
An easy computation shows that $$\sum_{i=1}^9 w_i^2\leq \frac{w^2}{9}+9\eps^2 \mbox{ and }\sum_{1\leq i\neq j\leq 9}(w_i^2 w_j)\leq \frac{8}{81}w^3+9w\eps^2.$$
Therefore
$$1-W(X)+W(R\cup R_{KM})\leq 1-w+\frac{w^2}{18}+\frac{8w^3}{81}+9(w+\frac{1}{2})\eps^2+w\eps^{p-1}+\delta.$$
Setting $w=\frac{3}{2}$, $\delta=\frac{1}{50}$ and $\eps=\frac{1}{100}$, we get
$1-W(X)+W(R\cup R_{KM})\leq 1-\frac{3}{2}+\frac{1}{8}+\frac{1}{3}+\frac{36}{10^4} +\frac{3}{200}+\frac{1}{50}<0$.
Thus, we proved the following claim:
\begin{Claim}
\label{R''}
Assume that $w=\frac{3}{2}$, $\delta=\frac{1}{50}$ and $\eps=\frac{1}{100}$.
Then
\begin{equation}
1-W(X)+W(R\cup R_{KM})<0.
\end{equation}
\end{Claim}
We are ready to finish the proof of Theorem~\ref{jaikin4} (so far for $p\geq 67$).
By Claim~\ref{R''} the presentation $(X,R\cup R_{KM})$ satisfies the GGS condition,
hence the group $\Gp_{abs}(X,R\cup R_{KM})$ is infinite. By definition,
$\Gp_{abs}(X,R\cup R_{KM})$ is a quotient of $\Gp_{abs}(X,R)$. On the other hand,
$\Gp_{abs}(X,R\cup R_{KM})$ is also a quotient of $\Gp_{abs}(X,R_{KM})$ and hence has property $(T)$.

\subsection{Proof of Theorem~\ref{jaikin4} in the general case.}
When $p$ is small, we do not have any examples of groups of the form
$\KMS(\dbF_p; n_1,\ldots, n_k)$ which are both Kazhdan and Golod-Shafarevich.
Instead, we will consider the groups
$\KMS(\dbF_q; n_1,\ldots, n_{9})$ with $q=p^{8}$, which have property $(T)$
by Theorem~\ref{GS_supply} since $p^{8}>8^2$.
Let $\calB=\{\lam_1=1,\ldots, \lam_{8}\}$ be a basis for $\dbF_q$ over $\dbF_p$.
For any $n_1,\ldots, n_{9}\in\dbN$ the group $\KMS(\dbF_q; n_1,\ldots, n_k)$
is given by the abstract presentation $\la X_{KM} |R_{KM}\ra$ where
\begin{align}
&\label{presq}
X_{KM}=\{x_{i,k}(\lam_j), 1\leq i\leq 9, 1\leq k\leq n_i, 1\leq j\leq 8\}\mbox{ and }&\\
&R_{KM}=R_{nilp}\cup R_{field} \mbox{ where }&\\
&R_{nilp}=\{x_{i,k}(\lam)^p, [x_{i,k}(\lam),x_{i,l}(\mu)],
\quad [x_{i,k}(\lam),x_{j,l}(\mu),x_{i,m}(\nu)]\},&\\
&\label{presq4}
R_{field}=\{[x_{i,k}(\lam),x_{j,l}(\mu)][x_{i,k}(\lam\mu),x_{j,l}(1)]^{-1}\},&
\end{align} and
$\lam,\mu,\nu$ are arbitrary elements of the basis $\calB$.

Here $x_{i,k}(\alpha)$
for $\alpha\in \dbF_q\setminus \calB$ is defined in the obvious way: write
$\alpha=\sum_{j=1}^8 c_j \lam_j$ with $0\leq c_j\leq p-1$ and set
$x_{i,k}(\alpha)=\prod_{j=1}^8 x_{i,k}(\lam_j)^{c_j}$.

\begin{proof}[of Theorem~\ref{jaikin4}]
Let $w=12$, fix $\delta>0$ and $\eps>0$ (to be specified at the very end),
and let $(X,R,W)$ be a triple satisfying $\Pres(12,\delta/2,\eps')$
with $\eps'=\min\{\eps,\delta/14\}$.

In order to find a common infinite quotient for $\Gp_{abs}(X,R)$
and the above KMS groups we need
to impose an additional condition on the triple $(X,R,W)$:

\begin{itemize}
\item[(CC)] For every $z\in\dbR$, the cardinality of the set $\{x\in X: W(x)=z\}$ is divisible by $8$.
\end{itemize}

This can be achieved by adding artificial generators $y_1,\ldots, y_m$
and relations $y_1=\ldots=y_m=1$ (that is, by replacing the presentation
$(X,R)$ with $(X\sqcup Y, R\sqcup Y)$ for some finite set $Y$),
and extending $W$ to a weight function on $(F(X\cup Y), X\cup Y)$
by assigning appropriate weights to elements of $Y$.
Clearly, for any $z\in \Im(W)$ we need to add at most $7$ generators $y$ with $W(y)=z$
to achieve (CC).
Thus, we can assume that $12\leq W(X\cup Y)\leq 12+7s$ and $W(R\cup Y)\leq W(R)+7s$ where
$s=\sum_{z\in \Im(W)}z$. By assumption $W(R)<\delta/2$ and $s<\eps'\leq \delta/14$, whence
$W(R)+7s< \delta$. Now contract the weight function $W$ on $(F(X\cup Y),X\cup Y)$ by the factor
$\frac{W(X\cup Y)}{12}\geq 1$. After relabeling $X\cup Y$ by $X$, $R\cup Y$ by $R$
and the new (contracted) weight function by $W$ we get that $(X,R,W)$
satisfies $\Pres(12,\delta,\eps)$ and condition (CC) holds.

Since $\max\{W(x): x\in X\}<\eps$, we can divide the set $X$ into $9$ subsets $X^1,\ldots, X^{9}$, such that
\begin{itemize}
\item[(a)] $|W(X^i)-W(X^j)|<8\eps$ for all $1\leq i,j\leq 9$
\item[(b)] For any $1\leq i\leq 9$ and $z\in\dbR$, the cardinality of the set $\{x\in X^i: W(x)=z\}$ is divisible by $8$.
\end{itemize}
For each $1\leq i\leq 9$ let $n_i=|X^i|/8$ (this is surely an integer by (b)),
and denote elements of $X^i$ by symbols $\{x_{i,k}(\lam_j)\}$ with $1\leq k\leq n_i$ and $1\leq \lam_j\leq 8$,
such that for each $i,k$ the weight $W(x_{i,k}(\lam_j))$ is the same for all $j$.
Now we can identify $X$ with the generating set $X_{KM}$ from the presentation
for $\KMS(\Fq; n_1,\ldots, n_{9})$ given by \eqref{presq}-\eqref{presq4},
and consider the presentation
$(X,R_{KM}\cup R)$. If we show that the pro-$p$ group $\Gp(X,R_{KM}\cup R)$ is infinite,
we will be done by the same argument as in the case $p\geq 67$.
Note that unlike the latter case the group $G=\Gp(X,R_{KM}\cup R)$ is not GGS
because of the relations $R_{field}$.
To prove that it is infinite, we will map the completed group algebra
$\Fp[[G]]$ onto a finite codimension subalgebra of certain $\Fq$-algebra, which, in turn,
will be GGS (as $\Fq$-algebra) and hence infinite.
\vskip .12cm
Let $U=\{u_{i,k}(\lam_j) : 1\leq i\leq 9, 1\leq k\leq n_i, 1\leq j\leq 8\},$
and consider the standard embedding of $F(X)$ into $\Fp\lla U\rra$ given by $x_{i,k}(\lam_j)\mapsto 1+u_{i,k}(\lam_j)$.
Let $S=\{r-1: r\in R_{KM}\cup R\}$ and $I$ the ideal of $\Fp\lla U\rra$ generated by $S$.
Recall from Section~2 that the algebra $\Fp[[G]]$ is isomorphic to $\Fp\lla U\rra/I$.

Now let $\Ugal$ be the set of formal symbols $\{u_{i,k} : 1\leq i\leq 9, 1\leq k\leq n_i\}$,
and consider the unique continuous $\Fp$-algebra homomorphism
$\phi: \Fp\lla U\rra\to \Fq\lla \Ugal\rra$ given by $$\phi(u_{i,k}(\lam_j))=\lam_j u_{i,k}.$$
Clearly, $\phi(\Fp\lla U\rra)$
contains the ideal of $\Fq\lla \Ugal\rra$ generated by $\Ugal$, so $\Im\phi$ is of finite codimension.
Let $S_1=\{\phi(r-1): r\in R\}$,  $S_2=\{\phi(r-1): r\in R_{nilp}\},$
and  $S_3=\{a^2b, aba, ba^2 : a,b\in \Ugal\}$. Let $\Igal$ be the ideal of $\Fq\lla \Ugal\rra$
generated by $S_1\cup S_2\cup S_3$.
We will prove that
\begin{itemize}
\item[(a)] $\phi(S)\subseteq \Igal$, whence there is a homomorphism
$\Fp\lla U\rra/I\to\Fq\lla \Ugal\rra/\Igal$ with image of finite codimension;
\item[(b)] The algebra $\Fq\lla \Ugal\rra/\Igal$ satisfies the GGS condition and hence it is infinite.
\end{itemize}
By earlier discussion, (a) and (b) would imply Theorem~\ref{thm:main1}.

For (a), we only need to show that $\phi(r-1)\in I$ for any $r\in R_{field}$
(for $r\in R_{nilp}\cup R$ we have $\phi(r-1)\in I$ by construction).
Any $r\in R_{field}$ is of the form $r=[x_{i,k}(\lam),x_{j,l}(\mu)][x_{i,k}(\lam\mu),x_{j,l}(1)]^{-1}$.
Let $a=\phi(u_{i,k})$ and $b=\phi(u_{j,l})$.
Then
\begin{multline*}
\phi([x_{i,k}(\lam),x_{j,l}(\mu)])=[1+\lam a,1+\mu b]=(1+\lam a)^{-1}(1+\mu b)^{-1}(1+\lam a)(1+\mu b)=\\
1+(1+\lam a)^{-1}(1+\mu b)^{-1}\lam\mu(ab-ba)
\equiv 1+\lam\mu(ab-ba) \mod \Igal.
\end{multline*}
Similarly, we have
$\phi([x_{i,k}(\lam\mu),x_{j,l}(1)])\equiv 1+\lam\mu(ab-ba) \mod \Igal$, whence
$$\phi(r-1)=\phi([x_{i,k}(\lam),x_{j,l}(\mu)]-[x_{i,k}(\lam\mu),x_{j,l}(1)])\phi([x_{i,k}(\lam\mu),x_{j,l}(1)])^{-1}\in \Igal.$$
Thus, we proved (a).

Now let $w$ be the weight function on $\Fp\lla U\rra$ which is $X$-compatible with the weight function $W$ on $(F(X),X)$.
Let $\wgal$ be the unique weight function on $\Fq\lla \Ugal\rra$ such that $\wgal(\phi(u))=w(u)$ for any $u\in U$ -- note that such that $\wgal$ exists since $w(u_{i,k}(\lam_j))=W(x_{i,k}(\lam_j))$ is independent of $j$.

By an obvious analogue of Lemma~\ref{easy} for algebras, to prove (b) it suffices to show that
$$1-\wgal(\Ugal)+\wgal(S_1)+\wgal(S_2)+\wgal(S_3)<0.$$
Arguing as in the proof of Lemma~\ref{weight2}(a), we conclude that 
$\wgal(\phi(f))\leq w(f)$ for any $f\in\Fp\lla U\rra$.
In particular, $\wgal(S_1)\leq w(\{r-1: r\in R\})= W(R)<\delta$.

For $1\leq i\leq 9$ let $U^i=\{u_{i,k}: 1\leq k\leq n_i\}$ and
$\wgal_i=\wgal(U^i)$, and set $\wgal=\sum \wgal_i$. By construction, $\wgal=\wgal(\Ugal)=w(U)/8=3/2$.
Also note that $|\wgal_i-\wgal_j|=\frac{|W(X^i)-W(X^j)|}{8}<\frac{8\eps}{8}=\eps$ for any $i,j$,
and thus $|\wgal_i-\frac{\wgal}{9}|<\eps$.

Since $\max\{w(u): u\in U\}=\max\{W(X): x\in X\}<\eps$, we have
$\wgal(S_3)\leq 3\eps \wgal^2$, and $\wgal(S_2)$
can be estimated as in the case $p\geq 67$.
We get
\begin{equation*}
\label{messy}
1-\wgal(\Ugal)+\wgal(S_1)+\wgal(S_2)+\wgal(S_3)\leq 1-\wgal+
\sum_{i=1}^{9}\wgal_i^2/2+\sum_{i\neq j} \wgal_i^2 \wgal_j+
\eps^{p-1} \wgal+3\eps \wgal^2+\delta.
\end{equation*}
Using the same estimates as in the case $p\geq 67$, we get
\begin{multline*}
1-\wgal(\Ugal)+\wgal(S_1)+\wgal(S_2)+\wgal(S_3)\leq \\
1-\wgal+\wgal^2/18+8\wgal^3/81 + 9(\wgal+\frac{1}{2})\eps^2 +
(\wgal+3\wgal^2)\eps+\delta<-\frac{1}{24}+36\eps^2+9\eps+\delta.
\end{multline*}
Setting $\delta=\frac{1}{50}$ and $\eps=\frac{1}{1000}$, we get
$1-\wgal(\Ugal)+\wgal(S_1)+\wgal(S_2)+\wgal(S_3)<0$,
which finishes the proof.
\end{proof}
\subsection{Proof of Theorem~\ref{thm:main0}}
In this subsection we establish the following result due to Jaikin-Zapirain:
\begin{Proposition}
\label{finiteindex}
Let $\Gamma$ be a finitely generated group and $H$ a
finite index subgroup $\Gamma$. If $H$ maps onto an
infinite group with property $(T)$, then $\Gamma$ also maps onto an infinite
group with property $(T)$.
\end{Proposition}

\begin{proof}
If $H$ maps onto an infinite group with property $(T)$, then so does
any finite index subgroup of $H$. Thus, we may assume that $H$
is normal in $\Gamma$. Let $N$ be a normal subgroup of $H$ such
that $H/N$ has property $(T)$. By \cite{Gr}, any finitely generated
abstract group has a just-infinite quotient. Thus,
replacing $N$ by a larger normal subgroup (if necessary),
we may assume that $H/N$ is just-infinite.

Let $S_0$ be a transversal of $H$ in $\Gamma$, and set $L=\cap_{s\in S_0} N^s$.
Then $L$ is a normal subgroup of $\Gamma$. We shall show that $\Gamma/L$ has
property $(T)$. Since $H/L$ is of finite index in $\Gamma/L$, it is
enough to show that $H/L$ has property $(T)$.

Note that $H/L$ naturally embeds in $\prod_{s\in S_0} H/N^s$.
Choose a subset $S\subseteq S_0$ such that the composite map
$$\pi:H/L\to\prod_{s\in S_0} H/N^s\to\prod_{s\in S} H/N^s$$
is injective and $S$ is minimal with this property.
For each $s\in S$ let $H_s=\pi(H/L)\cap (H/N^s)$. Since
$\pi(H/L)$ surjects onto each factor in $\prod_{s\in S} H/N^s$,
the group $H_s$ is normal in $H/N^s$. Since $H/N^s\cong H/N$
is just-infinite, $H_s$ is either trivial or of finite index in $H/N^s$.
But if $H_a$ is trivial for some $a\in S$, then $H/L$ injects in
$\prod_{s\in S\setminus\{a\}} H/N^s$, contrary to the minimality of $S$.
Thus, $H_s$ is of finite index in $H/N^s$ for each $s\in S$.
Therefore, $\pi(H/L)\supseteq \prod H_s$ is of finite index in $\prod_{s\in S_0} H/N^s$.

Since $H/N$ has property $(T)$, so do $\prod_{s\in S_0} H/N^s$ and its finite index subgroups.
Therefore, $H/L\cong \pi(H/L)$ also has property $(T)$.
\end{proof}

We completed the proof of Theorem~\ref{thm:main0} as it clearly
follows from Theorem~\ref{thm:main1} and Proposition~\ref{finiteindex}.
Recall that Theorem~\ref{thm:main1} was, in turn, a consequence of Proposition~\ref{deepdescent2}
and Theorem~\ref{jaikin4}. Thus, we actually established the following generalization of 
Theorem~\ref{thm:main0}:

\begin{Theorem}
\label{jaikin5}
Let $G$ be a generalized Golod-Shafarevich pro-$p$ group and
$\Gamma$ a dense countable subgroup of $G$. Then $\Gamma$
has an infinite quotient with property $(T)$.
\end{Theorem}

It is also natural to ask for a generalization (or rather strengthening)
of Theorem~\ref{thm:main0} of a different kind:

\begin{Question}
\label{q:JZ}
Let $\Gamma$ be a GGS group with respect to some prime $p$.
Does $\Gamma$ always have an infinite Kazhdan quotient which is also a GGS group?
\end{Question}

Note that for $p\geq 67$ the existence of such quotient for some
finite index subgroup of $\Gamma$ follows directly from the proof of Theorem~\ref{thm:main1}.
Unfortunately, the GGS condition is likely not satisfied by the
Kazhdan quotient for the entire $\Gamma$ constructed in the proof of Theorem~\ref{finiteindex}.

\appendix
\section{\\Uniform non-amenability of Golod-Shafarevich groups}
\centerline{by \sc Mikhail Ershov \rm and \sc Andrei Jaikin-Zapirain
\footnote{The author is supported by the Spanish Ministry of Science and Innovation, grant
 MTM2008-06680.}}
 \vskip .2cm

\markboth{\small\sc mikhail ershov and andrei jaikin-zapirain}{\small\sc uniform non-amenability of golod-shafarevich groups}

The main goal of this section is to prove that generalized Golod-Shafarevich groups
satisfy a strong form of non-amenability, called uniform non-amenability.
\vskip .1cm

{\bf Notation.} In the previous sections by
the standard abuse of notation we identified the generating set
$X$ of a (pro-$p$) presentation $(X,R)$ with its (canonical) image in
the pro-$p$ group $\Gp (X,R)$ or abstract group $\Gp_{abs}(X,R)$. In this section
such abuse of notation could lead to confusion. For this
reason, if $(X,R)$ is a presentation and $\Gamma=\Gp_{abs}(X,R)$,
we will denote the image of $X$ in $\Gamma$ by $X_{\Gamma}$.

\subsection{Basic definitions}

\begin{Definition}\rm Let $\Gamma$ be a finitely generated (abstract) group.
Given a finite generating set $S$ of $\Gamma$,
the \textit{F\o lner constant} \footnote{We thank Goulnara Arzhantseva for correcting our terminology here} 
$h(\Gamma,S)$ is defined by
$$h(\Gamma,S)=\inf_{A\subset Cay(\Gamma,S)}\frac{|\partial A|}{|A|}$$
where infimum is taken over all finite subsets $A$ of the Cayley
graph of $\Gamma$ with respect to $S$ and
$$\partial A=\{g\in A: gs\not\in A \mbox{ for some }s\in S\cup S^{-1}\}.$$
The group $\Gamma$ is called \textit{amenable} if
$h(\Gamma,S)=0$ for some (hence any) finite generating set $S$ of $\Gamma$.
\end{Definition}

\begin{Definition}\rm Let $\Gamma$ be a finitely generated group.
\begin{itemize}
\item[(a)] Given a unitary representation $V$ of $\Gamma$
and a generating set $S$ of $\Gamma$, we define
$\kappa(\Gamma,S;V)$ to be the largest $\eps\geq 0$
such that for any $v\in V$ there exists $s\in S$ with $\|sv-v\|\geq \eps\|v\|$.

\item[(b)] Given a generating set $S$ of $\Gamma$, the \textit{Kazhdan constant}
$\kappa(\Gamma,S)$ is defined to be the infimum of the set $\{\kappa(\Gamma,S;V)\}$
where $V$ runs over all unitary representations of $\Gamma$ without nonzero invariant vectors.

\item[(c)] The group $\Gamma$ is called a \textit{Kazhdan group}
(equivalently $\Gamma$ is said to have \textit{Kazhdan's property $(T)$})
if $\kappa(\Gamma,S)>0$ for some (hence any) finite generating set $S$ of $\Gamma$
\end{itemize}
\end{Definition}

Recall that to deduce Corollary~\ref{thm:nonam} from Theorem~\ref{thm:main0}
we used the fact that a group which is both amenable and
Kazhdan must be finite. This well-known fact follows, for instance, from the following
characterization of amenability:

\begin{Theorem}[\cite{Hu}]
\label{thm:hu}
Let $\Gamma$ be a group generated by a finite set $S$. Then
$\Gamma$ is amenable if and only if $\kappa(\Gamma,S;L^2({\Gamma}))=0$.
\end{Theorem}

The constant $\kappa(\Gamma,S;L^2({\Gamma}))$, which appears frequently
in subsequent discussion, will be denoted by $\alpha(\Gamma,S)$
and called \textit{Kazhdan $L^2$-constant}. Note that if
$\Gamma$ is infinite, $L^2({\Gamma})$ has no nonzero invariant vectors
whence
\begin{equation*}
\alpha(\Gamma,S)\geq \kappa(\Gamma,S).
\end{equation*}
We finish this subsection with several inequalities involving F\o lner and Kazhdan constants.

\begin{Lemma}
\label{Arzh}
Let $\Gamma$ be a group generated by a finite set $S$. If $\pi:\Gamma\to \Gamma'$
is a surjective homomorphism and $S'=\pi(S)$, then
\begin{itemize}
\item[(a)] $h(\Gamma,S)\geq h(\Gamma',S')$
\item[(b)] $\alpha(\Gamma,S)\geq \alpha(\Gamma',S')$ and $\kappa(\Gamma,S)\leq \kappa(\Gamma',S')$
\end{itemize}
\end{Lemma}
\begin{proof} (a) appears as Theorem~4.1 in \cite{A+} and (b) is Lemma~3.4 in \cite{Os1}.
\end{proof}

\begin{Definition}[\cite{Os3}]\rm Let $\Gamma$ be a group generated
by a finite set $S$. Given a finite subset $Y$ of $\Gamma$, define
$\depth_S(Y)$ to be the minimal $L\in\dbZ_{\geq 0}$ such that every element of $Y$
can be expressed by a word of length $\leq L$ in $S\cup S^{-1}$.
\end{Definition}

\begin{Lemma}
\label{T_bg}
Let $\Gamma$ be a group generated by a finite set $S$ and $\Delta$
a subgroup of $\Gamma$ generated by a finite set $Y$.
Then for any unitary representation $V$ of $\Gamma$ we have
$$\kappa(\Gamma,S;V)\geq \frac{\kappa(\Delta,Y;V)}{\depth_S(Y)}$$
\end{Lemma}
\begin{proof} This is very well known, but since the proof is very short
we give it here.
Let $L=\depth_S(Y)$ and take any $\kappa>\kappa(\Gamma,S;V)$.
Then there is $v\in V$ with $\|sv-v\|<\kappa \|v\|$ for all $s\in S\cup S^{-1}$
(since $\|sv-v\|=\|s^{-1}v-v\|$ by unitarity).
Now given $y\in Y$, write $y=s_1\ldots s_n$ with $s_i\in S\cup S^{-1}$ and $n\leq L$.
We have
\begin{multline*}
\|s_1 \ldots s_n v-v\|\leq \|s_1 \ldots s_n v-s_1 \ldots s_{n-1} v\| +\ldots
+\|s_1 s_2 v- s_1 v\| +\|s_1 v-v\|\\=\sum_{i=1}^n \|s_i v-v\|< L\kappa \|v\|,
\end{multline*}
and thus $\kappa(\Delta,Y;V)<L\kappa$.
\end{proof}

\begin{Lemma}
\label{Arzh2}
Let $\Gamma$ be a group generated by a finite set $S$ and $\Delta$
a subgroup of $\Gamma$ generated by a finite set $Y$. Let $L=\depth_S(Y)$. Then
$$\mbox{\rm (a) }h(\Gamma,S)\geq \frac{1}{|Y|L+1}h(\Delta,Y);\quad\quad
\mbox{\rm (b) }\alpha(\Gamma,S)\geq \frac{\alpha(\Delta,Y)}{\sqrt{|Y|}L}.$$
\end{Lemma}
\begin{proof}
(a) appears as Theorem~7.1 in \cite{A+}, and (b) is the corrected version
of Lemma~2.9 in \cite{Os3}. For completeness, we shall
sketch the proof of (b).
We start with a general claim:

\begin{Claim} Let $\{V_i\}_{i=1}^{\infty}$ be unitary representations
of a group $G$ generated by a finite set $X$ and $\kappa=\inf\limits_{i\geq 1}\{\kappa(G,X; V_i)\}$.
Let $V=\widehat\oplus_{i=1}^{\infty} V_i$ (where $\widehat\oplus$
denotes completed direct sum). Then
$$\kappa(G,X; V)\geq \frac{\kappa}{\sqrt{|X|}}.$$
\end{Claim}

\begin{proof}[of the Claim] Take any $v\in V$
and write $v=\sum_{i=1}^{\infty} v_i$ with $v_i\in V_i$.
By definition of $\kappa$ there exists a function $s:\dbN\to X$
such that for each $i\in\dbN$ we have $\|s(i) v_i - v_i\|\geq \kappa \|v_i\|$.
Thus we have
\begin{multline*}
\sum_{x\in X}\|x v -v\|^2=\sum_{x\in X}\sum_{i=1}^{\infty}\|x v_i-v_i\|^2\geq \\
\sum_{i=1}^{\infty}\|s(i) v_i-v_i\|^2\geq \sum_{i=1}^{\infty} \kappa^2 \|v_i\|^2=\kappa^2 \|v\|^2.
\end{multline*}
Thus, for some $x\in X$ we must have $\|x v -v\|^2\geq \frac{\kappa^2}{|X|} \|v\|^2$.
\end{proof}

We proceed with the proof Lemma~\ref{Arzh2}(b). Let $V=L^2(\Gamma)$. It is easy to see that
as a $\Delta$-module $V=\widehat\oplus_{\gamma\in \Gamma/\Delta} V_{\gamma}$
with each $V_{\gamma}\cong L^2(\Delta)$. Then $\kappa(\Delta,Y;V_{\gamma})=\alpha(\Delta,Y)$
for each $\gamma$ and hence $\kappa(\Delta,Y;L^2(\Gamma))\geq \frac{\alpha(\Delta,Y)}{\sqrt{|Y|}}$
by the above claim. By Lemma~\ref{T_bg} we have
$\alpha(\Gamma,S)=\kappa(\Gamma,S;L^2(\Gamma))\geq \frac{\kappa(\Delta,Y;L^2(\Gamma))}{L}$
which finishes the proof.
\end{proof}
\begin{Remark} The above argument essentially follows the proof of Lemma~2.9
in \cite{Os3}. In fact, \cite[Lemma~2.9]{Os3} asserts
the stronger inequality $\alpha(\Gamma,S)\geq \frac{\alpha(\Delta,Y)}{L}$;
however, the proof implicitly contains unjustified claim that the function
$s:\dbN\to X$ defined above may be chosen constant. We note that this correction
does not affect the validity of any other results in \cite{Os3}.
\end{Remark}

\subsection{Uniform non-amenability}

In view of Theorem~\ref{thm:hu}, given a non-amenable group $\Gamma$, one can measure
the ``extent'' of its non-amenability using either F\o lner constants
or Kazhdan $L^2$-constants. This suggests two possible
definitions of \textit{uniformly non-amenable groups}.
Define the \textit{uniform F\o lner constant} $h(\Gamma)$
and the \textit{uniform Kazhdan $L^2$-constant} $\alpha(\Gamma)$ by
$$h(\Gamma)=\inf_S h(\Gamma,S)\quad \mbox{ and }\quad \alpha(\Gamma)=\inf_S \alpha(\Gamma,S),$$
where infimum is taken over all finite generating sets of $\Gamma$.

In \cite{Os3}, a group $\Gamma$ is called uniformly non-amenable
if $\alpha(\Gamma)>0$, and in \cite{A+} a group $\Gamma$ is called
uniformly non-amenable if $h(\Gamma)>0$. Lemma~\ref{Arzh1}
below shows that uniform non-amenability in the sense of \cite{Os3}
implies uniform non-amenability in the sense of \cite{A+}. To the
best of our knowledge, it is an open question whether the
converse implication holds. We shall talk about uniform non-amenability
in the (stronger) sense of \cite{Os3}.

\begin{Lemma}\rm(see e.g. {\cite[Proposition~2.4]{A+}})
\footnote{Formally, {\cite[Proposition~2.4]{A+}} only states the
inequality between the uniform F\o lner and Kazhdan $L^2$-constants,
but the proof actually establishes the full statement of Lemma~\ref{Arzh1}}
\label{Arzh1}
\textit{
Let $\Gamma$ be a group. Then for any finite generating set $S$ of $\Gamma$
we have $h(\Gamma,S)\geq \frac{1}{2}\alpha(\Gamma,S)^2$. In particular,
$h(\Gamma)\geq \frac{1}{2}\alpha(\Gamma)^2.$}
\end{Lemma}

The first examples of non-amenable but not uniformly non-amenable groups were constructed
by Osin in \cite{Os2} -- this gave an answer to a question of Shalom. In recent years
many important classes of groups were shown to be uniformly non-amenable:
these include finitely generated linear groups with non-abelian free subgroups \cite{Br}, (non-elementary)
hyperbolic groups and free Burnside groups of sufficiently large odd exponent \cite{Os3}.

Lemma~\ref{Arzh2}(b) easily implies that  a finitely generated group $\Gamma$ is uniformly non-amenable whenever the following condition is satisfied: there exists a non-amenable group $\Lambda$ and a fixed finite generating set $T$ of $\Lambda$ such that given any finite generating set $S$ of $\Gamma$ there exists an embedding
$\iota_S: \Lambda\to \Gamma$ with $\depth_S(\iota_S(T))\leq N$ for some $N$ independent of $S$.
In fact, this is a typical way to prove non-uniform amenability.

Our proof of uniform non-amenability of Golod-Shafarevich groups will implicitly
use a statement of this kind, but we will not be able to satisfy the above
condition for a fixed pair $(\Lambda,T)$. The main ingredient in our proof
will be the following quantitative version of Theorem~\ref{jaikin4}.

\begin{Theorem}
\label{jaikin6}
Let $(X,R,W)$ be a weighted presentation satisfying
$\Pres(12,1/100,1/1000)$ and $\Gamma=\Gp_{abs}(X,R)$. Then there exists an infinite group
$\Omega$ with $(T)$ and a surjective homomorphism  $\pi:\Gamma\to \Omega$ such that
\begin{equation}
\label{eq:kappa}
\kappa(\Omega,\pi(X_{\Gamma}))\geq \frac{1}{25|X|}.
\end{equation}
In particular, $\alpha(\Gamma,X_{\Gamma})\geq \frac{1}{25|X|}.$
\end{Theorem}

Finally, we remark that an important consequence of uniform non-amenability
is uniform exponential growth, and for many classes of groups the easiest
way to establish uniform exponential growth is by proving their uniform
non-amenability. However, this is not the case for the class of Golod-Shafarevich
groups where uniform exponential growth has already been known -- this
result was pointed out by Bartholdi and Grigorchuk~\cite{BaGr}, with the proof based on
\cite[Lemma~8]{Gr2}

\subsection{Proof of uniform non-amenability for GGS groups}

\begin{proof}[of Theorem~\ref{jaikin6}]
First note that the desired lower bound on $\alpha(\Gamma,X_{\Gamma})$
indeed follows from \eqref{eq:kappa} since
$$\alpha(\Gamma,X_{\Gamma})\geq \alpha(\Omega,\pi(X_{\Gamma}))\geq \kappa(\Omega,\pi(X_{\Gamma})).$$
Here the first inequality holds by Lemma~\ref{Arzh}(b), and the second one holds
since $\Omega$ is infinite.

Inequality \eqref{eq:kappa} follows easily from the analysis of the proof of Theorem~\ref{jaikin4}.
For simplicity, we shall only discuss the case $p\geq 67$; the case $p<67$ is similar.

Recall that any group of the form $\KMS(\Fp,\{n_1,\ldots, n_9\})$
has a generating set $S=\sqcup_{i=1}^9 S_i$ where each $S_i$ consists of $n_i$
pairwise-commuting elements of order $p$ and $\la S_i\ra=\calU_i$, the $i^{\rm th}$
root subgroup of $\KMS(\Fp,\{n_1,\ldots, n_9\})$.

In the proof of Theorem~\ref{jaikin4} we established the following:
there exist a group $\Lambda=\KMS(\Fp,\{n_1,\ldots, n_9\})$ with $\sum n_i=|X|$, an
infinite group $\Omega$ and surjective homomorphisms $\pi:\Gamma\to\Omega$
and $\theta:\Lambda\to \Omega$ such that
$\pi(X_{\Gamma})= \theta(S)$ where $S$ is a generating set of $\Lambda$
of the above form.

Since $p>(9-1)^2$, \cite[Corollary~7.2]{EJ} implies that $\Lambda$ has property $(T)$ and yields
the following bound for the Kazhdan constant:
$$\kappa(\Lambda,\cup \calU_i)\geq \sqrt{\frac{2}{9}\left(1-\frac{9-1}{\sqrt{p}}\right)}> \frac{1}{25}.$$
By the assumption on $S$ we have $\depth_{S}(\cup \calU_i)\leq p\cdot \max\{n_i\}< p|X|$,
and thus Lemma~\ref{T_bg} yields 
\begin{equation}
\label{bsetup}
\kappa(\Lambda,S)\geq \frac{\kappa(\Lambda,\cup \calU_i)}{p|X|}\geq \frac{1}{25 p|X|}.
\end{equation}
Since $\kappa(\Omega,\pi(X_{\Gamma}))=\kappa(\Omega,\theta(S))\geq \kappa(\Lambda,S)$
by Lemma~\ref{Arzh}(b), the proof is complete.
\end{proof}

\begin{Definition}\rm Let $G$ be a pro-$p$ group and $S$ a generating set for $G$.
Let $(X,R)$ be a presentation for $G$ and
$\pi:F(X)\to G$ the natural surjection. We will say that $(X,R)$
is a \textit{presentation for the pair $(G,S)$} if $\pi(X)=S$ (note that we do not
require that $\pi$ is injective on $X$).
\end{Definition}

Recall that for a finitely generated pro-$p$ group $G$
we denote by $\Phi(G)$ the Frattini subgroup of $G$.

\begin{Lemma}
\label{jaikin1}
Let $H$ be a finitely generated pro-$p$ group and $S=\{s_1,\ldots, s_d\}$
a generating set for $H$. Let $(X,R)$ be a presentation for
the pair $(H,S)$, and let $W$ be a weight function on $(F(X),X)$.
Then there exists an open normal subgroup $K$ of $H$ with the following property:
if $\Sgal=\{\sgal_1,\ldots,\sgal_d\}$ is a subset of $H$
such that $$\sgal_i\equiv s_i\mod K,$$ then $\Sgal$ generates $H$ and
there is a presentation $(X,\Rgal)$ for the pair $(H,\Sgal)$ such that
$W(\Rgal)=W(R)$.
\end{Lemma}
\begin{proof} Let $F=F(X)$ and $\pi: F\to H$ the natural surjection.
Let $\delta=\min\{W(x): x\in X\}$,
$$O=\{f\in F : W(f)<\delta\} \mbox{ and } K=\pi(O).$$
Clearly, $O$ is an open normal subgroup of $F$, and hence $K$ is open and normal in $H$.
We will show that $K$ has the required property.

We are given that $\pi(X)=S$. Let $e=|X|$ and $\{x_1,\ldots, x_e\}$ the elements of $X$.
By assumption on $\Sgal$ we can construct a subset $\Xgal=\{\xgal_1,\ldots,\xgal_e\}\subset F$
such that $\xgal_i\equiv x_i\mod O$ for $1\leq i\leq e$
and $\pi(\Xgal)=\Sgal$.

It is clear that $O\subseteq\Phi(F)$. Thus, $\Xgal$ generates $F$ modulo $\Phi(F)$,
and so $\Xgal$ is a generating set for $F$; furthermore, $\Xgal$ is a free generating set for
$F$ since $|\Xgal|=|X|$. It follows that there exists an isomorphism
$\theta: F\to F$ such that $\theta(x_i)=\xgal_i$ for $1\leq i\leq e$.
It is easy to show that
\begin{equation}
\label{degpreserve}
W(\theta(f))=W(f) \mbox { for any } f\in F
\end{equation}
(this is ensured by the definition of $O$ and the choice of $\Xgal$).

Now consider the map $\widetilde\pi=\pi\theta: F\to H$.
Note that $\widetilde\pi(X)=\pi(\Xgal)=\Sgal$ (so $\Sgal$ generates $H$)
and the set $\Rgal=\theta^{-1}(R)$ generates $\Ker\widetilde\pi$ as a (closed)
normal subgroup of $F$. Thus, $(X,\Rgal)$ is a presentation for $(H,\Sgal)$,
and the equality $W(\Rgal)=W(R)$ holds by \eqref{degpreserve}.
\end{proof}

We are now ready to prove uniform non-amenability of GGS groups.
As in Theorem~\ref{jaikin5}, we prove a slightly
more general result.

\begin{Theorem}
\label{uniformna}
Let $G$ be a generalized Golod-Shafarevich pro-$p$ group. Then
any dense finitely generated subgroup of $G$ is uniformly non-amenable.
\end{Theorem}
\begin{proof}By Theorem~\ref{deepdescent}
some open subgroup $H$ of $G$ has a weighted presentation $(X,R,W)$
satisfying $\Pres(12,1/100,1/1000)$.
Let $K$ be an open normal subgroup of $H$ satisfying the conclusion of Lemma~\ref{jaikin1}.

Let $S$ be the image of $X$ in $H$,
so that $(X,R)$ is a presentaiton for $(H,S)$.
Let $\{s_1,\ldots,s_d\}$ be the elements of $S$.

Now let $\Gamma$ be any dense finitely generated subgroup of $G$
and $Q$ any finite generating set of $\Gamma$. By density of $\Gamma$ there exists a subset 
$\Sgal=\{\sgal_1,\ldots,\sgal_d\}$ of $\Gamma$ such that
\begin{itemize}
\item[(i)] $\sgal_i\equiv s_i\mod K$ for $1\leq i\leq d\,\,$ (so in particular $\Sgal\subset H$).
\item[(ii)] $\depth_{Q}(\Sgal)\leq [G:K]$.
\end{itemize}
Condition (i) and Lemma~\ref{jaikin1} imply that there is a presentation $(X,\Rgal)$
for the pair $(H,\Sgal)$ with $W(\Rgal)=W(R)$. In particular, the triple $(X,\Rgal,W)$
satisfies $\Pres(12,1/100,1/1000)$ since $(X,R,W)$ has this property.

Let $\Lambda$ be the abstract subgroup of $\Gamma$ generated by $\Sgal$. By the
choice of $\Rgal$ we can identify $\Lambda$ with the group $\Gp_{abs}(X,\Rgal)$
so that $X_{\Lambda}=\Sgal$. Then by Theorem~\ref{jaikin6} we have
$$\alpha(\Lambda,\Sgal)\geq \frac{1}{25p|X|}.$$

On the other hand, by condition (ii) and Lemma~\ref{Arzh2}(b) we have
$$\alpha(\Gamma,Q)\geq \frac{\alpha(\Lambda,\Sgal)}{[G:K]\sqrt{|\Sgal|}}\geq
\frac{\alpha(\Lambda,\Sgal)}{[G:K]\sqrt{|X|}}.$$
Thus, $\alpha(\Gamma,Q)$ is bounded below by $\frac{1}{25p|X|^{3/2}[G:K]}$,
the quantity which does not depend on $Q$ (in fact, it does not even depend on $\Gamma$).
This inequality finishes the proof.
\end{proof}

\subsection{Explicit bound for uniform Kazhdan $L^2$-constants in GS groups}

Let $(X,R)$ be a presentation satisfying the GS condition.
\footnote{Here we indeed restrict our considerations from GGS groups to GS groups.
This restriction does not seem to be essential, but it does simplify some arguments}
Let $G=\Gp(X,R)$ and $\Gamma$ a dense finitely generated subgroup of $G$.
In this subsection we shall obtain an explicit lower bound for the Kazhdan
$L^2$-constant $\alpha(\Gamma)$ which depends only on the Hilbert series $H_X(t)$ and $H_R(t)$
(with respect to the standard degree function on $F(X)$).
In the course of the proof we estimate from above the (finite) index of a subgroup
$\Delta$ of $\Gamma$ to which the proof of Theorem~\ref{thm:main1} applies.
Such estimate is of independent interest, especially in view of
the remark following Question~\ref{q:JZ}.

\begin{Theorem}
\label{thm:Cheeger} Let $(X,R)$ be a presentation satisfying the GS condition and $G=\Gp(X,R)$.
Choose $0<t_1<t_0<1$ such that $1-H_{X}(t_i)+H_{R}(t_i)<0$ for $i=0,1$, and
set $$\mu=-(1-H_{X}(t_0)+H_{R}(t_0)),\quad \rho=\frac{t_0}{t_1}.$$
Fix $M\geq 6\cdot 10^4$, and set 
$$k_0=]\frac{\log(M)-\log(\mu)}{\log(\rho)}[\quad\mbox{ and }\quad N=\sum_{i=0}^{k_0} |X|^i.$$
The following hold:

\begin{itemize}
\item[(a)] There exists a subgroup of $G$ of index at most $p^{N}$
which has a weighted presentation $(X',R',W)$ with $1-W(X')+W(R')<-M$.

\item[(b)] Let $\Gamma$ be a dense finitely generated subgroup of $G$. Then
there exists  $C>0$ which depends only on $p,N, |X|$ (and can be computed explicitly)
such that $\alpha(\Gamma)\geq C.$
\end{itemize}
\end{Theorem}
\begin{proof}
To prove (a) we analyze the proof of Theorem~\ref{deepdescent3}.
Let $W$ be the $(D,t_0)$-weight function on $(F(X),X)$, where $D$ is the standard
degree function. For $k\in\dbN$ let $(X_k,R_k)$ be defined as in the proof of
Theorem~\ref{deepdescent} and $\mu_k=-(1-W(X_k)+W(R_k))$.
\begin{Claim}
\label{claim_a} $\mu_{k_0}\geq M$.
\end{Claim}
\begin{proof}[of Claim~\ref{claim_a}]
Let $A=\Fp[[G]]$, let $\{A_n\}$ be the $D$-filtration of $A$ and $a_n=\dim A_n/A_{n+1}$.
Let $\{c_n\}$ be as in the proof of Theorem~\ref{deepdescent}.
Then $$\mu_k\geq  \mu\prod_{i=1}^{k} \left(\frac{1-t_0^{pi}}{1-t_0^i}\right)^{c_i}\geq
\mu \sum_{i=0}^{k}a_i t_0^i.$$
where the second inequality holds by Proposition~\ref{GGS2}.
Since $1-H_{X}(t_1)+H_R(t_1)<0$, by the Golod-Shafarevich theorem the series
$\sum_{i=0}^{\infty}a_i t_1^i$ diverges and thus
\begin{equation}
\label{eq:limsup}
\limsup_{n\to\infty}\sqrt[n]{a_n}\geq \frac{1}{t_1}.
\end{equation}

Since $\{A_n\}$ is the filtration of $A$ corresponding to the standard degree function,
for any $n,m\in\dbN$ we have $(A_n/A_{n+1}) \cdot (A_m/A_{m+1})=A_{n+m}/A_{n+m+1}$ in $gr(A)$,
and thus $a_{n+m}\leq a_n a_m$. This observation and \eqref{eq:limsup} easily imply that
$a_n\geq (1/t_1)^n$ for any $n$. Therefore,
$$\mu_{k_0}\geq \mu \sum_{i=0}^{k_0}a_i t_0^i\geq \mu (t_0/t_1)^{k_0}=\mu \rho^{k_0}\geq M\geq 6\cdot 10^4.$$
\end{proof}
\vskip .1cm
By Claim~\ref{claim_a} the presentation $(X',R')=(X_{k_0}, R_{k_0})$ satisfies the inequality
$1-W(X')+W(R')<-M$. By construction $\Gp(X_{k},R_k)$ is the $(k+1)^{\rm st}$ term of the Zassenhaus
$p$-series of $\Gp(X,R)$, whence $\log_p[\Gp(X,R):\Gp(X_k,R_k)]\leq \sum_{i=0}^{k} |X|^i$
for any $k$. In particular, $\log_p[\Gp(X,R):\Gp(X',R')]\leq \sum_{i=0}^{k_0} |X|^i=N$,
which proves (a).
\vskip .1cm

(b) We start by analyzing the proof of Theorem~\ref{uniformna}.
By Theorem~\ref{deepdescent4}(b) the group $\Gp(X',R')$ has a weighted presentation
$(X'',R'',W'')$ satisfying $\Pres(12,1/100,1/1000)$. Thus, in the proof of
Theorem~\ref{uniformna} we can set $H=\Gp(X',R')$.

Now let $K$ be an open normal subgroup of $H$ which satisfies the conclusion of 
Lemma~\ref{jaikin1} applied to the weighted presentation $(X'',R'',W'')$.
We choose $K$ so that the index $[G:K]$ is minimal possible.
In the proof of Theorem~\ref{uniformna} we showed that
$\alpha(\Gamma)\geq \frac{1}{25p|X''|^{3/2}[G:K]}$. Thus, it remains to find upper
bounds for $|X''|$ and $[G:K]$ in terms of $p$, $|X|$ and $N$.
Since we already have such a bound for $[G:H]$, it is enough to find a bound for $[H:K]$.

Given a finite set $\Xgal$ and a weight function $\Wgal$ on $(F(\Xgal),\Xgal)$,
we set  $w_{min}(\Xgal,\Wgal)=\min\{\Wgal(x): x\in\Xgal\}$,
$w_{max}(\Xgal,\Wgal)=\max\{\Wgal(x): x\in\Xgal\}$ and
$\theta(\Xgal,\Wgal)=\frac{\log(w_{max})}{\log(w_{min})}$.
Using the proof of Lemma~\ref{jaikin1} it is easy to show that
$$\log_p[H:K]\leq \sum_{i=0}^{m}|X''|^{m} \mbox{ where }m=]\theta(X'',W'')[.$$
Thus, it remains to find upper bounds for $|X''|$ and $\theta(X'',W'')$.

Recall that the weighted presentation $(X'',R'',W'')$ is obtained from $(X,R,W)$
using a sequence of $p$-descents, good changes of generators and relators,
cleanups and weight contractions. It is easy to see that
\begin{itemize}
\item[(i)] If $(\Xgal,\Rgal,\Wgal)\to (\Xgal',\Rgal',\Wgal)$ is a $\Wgal$-good change of generators
or relators or a cleanup, then $|\Xgal'|\leq |\Xgal|$ and $\theta(\Xgal',\Wgal')\leq \theta(\Xgal,\Wgal)$.
\item[(ii)] If $(\Xgal,\Rgal,\Wgal)\to (\Xgal',\Rgal',\Wgal)$ is a $p$-descent, then
$|\Xgal'|=p(|\Xgal|-1)+1<p|\Xgal|$ and $\theta(\Xgal',\Wgal)\leq p\theta(\Xgal,\Wgal)$.
\item[(iii)] If $(\Xgal,\Rgal,\Wgal)\to (\Xgal,\Rgal,\Wgal')$ is a weight contraction,
then $\theta(\Xgal,\Wgal')\leq \theta(\Xgal,\Wgal)$.
\end{itemize}

Since a $p$-descent replaces the group defined by a presentation by a subgroup
of index $p$ and all other operations listed above do not modify the group, the number of $p$-descents
in the sequence is equal to $\log_p[\Gp(X,R):\Gp(X'',R'')]=\log_p[G:H]\leq N$.
Also observe that $\theta(X,W)=1$ (since by $W(x)=t_0$ for all $x\in X$ by construction).
Using (i),(ii) and (iii), we conclude that $|X''|<p^N |X|$ and $\theta(X'',W'')\leq p^N$.
These are desired bounds for $|X''|$ and $\theta(X'',W'')$. The proof is complete.
\end{proof}

\section{On subgroup growth of Golod-Shafarevich groups}
\centerline{by \sc Andrei Jaikin-Zapirain \footnote{The author is supported by the
Spanish Ministry of Science and Innovation, grant
 MTM2008-06680.}}
 \vskip .2cm

\markboth{\small\sc andrei jaikin-zapirain} {on subgroup growth of golod-shafarevich groups}

 As mentioned in the introduction, there are many results
indicating that Golod-Shafarevich groups are ``large". One of the
examples is the main theorem of this paper. We add one more
property to this list: we will show that Golod-Shafarevich groups
have many subgroups of finite index.

If $G$ is an abstract group, we denote by $a_m(G)$ the number of subgroups
of index $m$ in $G$.  Similarly, if $G$ is a profinite group, $a_m(G)$
denotes the number of open subgroups
\footnote{If $G$ is a finitely generated profinite group,
by an important result of Nikolov and Segal \cite{NS} every subgroup of finite
index in $G$ is automatically open} of index $m$ in $G$. If $G$
is finitely generated (abstract or profinite) then $a_m(G)$ is finite for all $m$.
Note that the number of subgroups of index $m$ in an abstract group
$G$ coincides with the number of open subgroups of index $m$ in the profinite
completion $\widehat G$ of $G$: $a_m(G)=a_m(\Ghat)$.

Proofs of the following results on subgroup growth can be found in \cite[Chapters~1-3]{LS}.
For a free group $F_d$ on $d\geq 2$ generators we have $(m!)^{d-1}\le a_m(F_d)\le m(m!)^{d-1}$, and
$a_m(G)\le a_m(F_d)$ for any $d$-generated abstract group $G$.
The same holds for profinite groups: If $\hat F_d$ denotes a free profinite group on $d$
generators, then $(m!)^{d-1}\le a_m(\hat F_d)\le m(m!)^{d-1}$  and $a_m(G)\le
a_m(\hat F_d)$ for any $d$-generated profinite group $G$. For a
free pro-$p$ group $(F_d)_{\hat p}$ on $d$ generators we have
$a_m((F_d)_{\hat p})\le 2^{md}$. From this we may deduce that for
any finitely generated pro-$p$ group $G$ we have
$\log_2(\log_2(a_m(G)))\leq \log_2(m)+C$ for some constant $C$
(of course $a_m(G)=0$ if $m$ is not a power of $p$).
It is also known that there is a constant $C'$
such that $\log_2(\log_2(a_m((F_d)_{\hat p})))\geq \log_2(m)+C'$
if $m=p^k$ for some $k$.

A well-known characterization of  $p$-adic analytic pro-$p$ groups says that
a pro-$p$ group $G$ is $p$-adic analytic if and only if
there is a constant $C$ such that 
$$\log_2(\log_2 (a_m(G)))\leq \log_2(\log_2 m)+C \mbox{ for all } m.$$ 
By a result of A. Shalev \cite{Sh}, if a pro-$p$ group $G$ is not $p$-adic analytic, there
is a constant $C$ such that
$$\log_2(\log_2( a_m(G)))\geq 2\cdot\log_2(\log_2 m)+C \mbox{ for infinitely many values of }m. $$ 
In an unpublished work A. Shalev (see \cite[Theorem~4.6.4]{LS}) showed
that if $G$ is a Golod-Shafarevich pro-$p$ group, then for any $\eps>0$
we have 
$$\log_2(\log_2 (a_m(G)))\geq (3-\eps) \log_2(\log_2 m) \mbox{ for infinitely many values of } m.$$ 
The following result significantly improves this bound:

\begin{Theorem}
\label{subgpgrowth} Let $G$ be  a generalized Golod-Shafarevich
pro-$p$ group. Then there exists a constant $\beta=\beta(G)>0$ such that
$$\log_2 (\log_2 (a_m(G)))\ge (\log_2 m)^{\beta}$$ for infinitely many
values of $m$.
\end{Theorem}

\markboth{}{\small\sc on subgroup growth of golod-shafarevich groups}

\begin{proof}
Let $(X,R)$ be a presentation of $G$ satisfying the GGS
condition for some degree function $D$ on $(F(X),X)$ and $t_0\in
(0,1)$. Let $\{G_n\}$ be the $D$-filtration of $G$ and put
$$c_n=\dim_{\mathbb F_p}G_n/G_{n+1}=\log_p[G_n:G_{n+1}].$$ Let
$t_1=\frac 1 {\limsup \sqrt[n]c_n}$. Since $t_1$ coincides with
the radius of convergence of $Hilb_{D,\F_p[[G]]}(t)$ (for instance,
this follows easily from Proposition~\ref{GGS2}), we have $t_0>t_1$.

Let $\eps=\min\{\frac {t_1}2,\frac{t_0-t_1}2\}$. For infinitely
many $n$ we have $c_{n-1}\ge (\frac 1{t_1+\eps} )^{n-1}$ and
$|G/G_n|\le p^{(\frac 1{t_1-\eps})^{n-1}}$.  We fix $n$
satisfying these two conditions.

Let $(X',R')$ be the presentation of $K=G_n$ obtained in Theorem~\ref{descseq}.
Observe first that $d(G_n)\ge H_{X'}(t_0)-H_{R'}(t_0)$, where
$d(G_n)$ is the minimal number of generators of $G_n$. Indeed, let
$m=d(G_n)$ and let $\{y_1,.., y_m\}\subseteq X'$ be a generating
set for $G_n$. Put $Y=\{y_1,.., y_{m-1}\}$ and consider the group
$H=\Gp(X', R'\cup Y)$. This group is (pro)cyclic and in
particular it is not a generalized Golod-Shafarevich group. Hence
$$1-H_{X'}(t_0)+H_{R'}(t_0)+H_Y(t_0)\ge 0.$$ On the other hand,
$H_Y(t_0)\le (m-1)t_0 \le m-1=d(G_n)-1$, whence $d(G_n)\ge
H_{X'}(t_0)-H_{R'}(t_0)$.

If we substitute $t_0$ in the expression \eqref{keyformula} of Theorem~\ref{descseq}, 
we obtain that
$$d(G_n)-1\ge H_{X'}(t_0)-H_{R'}(t_0)-1\ge
(H_{X}(t_0)-H_{R}(t_0)-1)\prod_{i=1}^{n-1}\left(\frac{1-t_0^{pi}}{1-t_0^i}\right)^{c_i}.$$
Note that $\prod_{i=1}^{n-1}(\frac{1-t_0^{pi}}{1-t_0^i})^{c_i}\ge
(1+t_0^{n-1})^{c_{n-1}}$, and $c_{n-1}\ge (\frac 1{t_1+\eps} )^{n-1}$
by our choice of $n$. Therefore,

$$\prod_{i=1}^{n-1}\left(\frac{1-t_0^{pi}}{1-t_0^i}\right)^{c_i}\ge [(1+t_0^{n-1})^{\frac1{t_0^{n-1}}}]^{(\frac
{t_0}{t_1+\eps})^{n-1}}\ge  2^{(\frac {t_0}{t_1+\eps})^{n-1}}.$$
Since  $(H_{X}(t_0)-H_{R}(t_0)-1)>0$ and $\frac
{t_0}{t_1+\eps}>1$,  there exists $\alpha>0$ (independent of $n$) such that
$$\begin{array}{lll}
d(G_n)-1 &\ge& (H_{X}(t_0)-H_{R}(t_0)-1)2^{(\frac
{t_0}{t_1+\eps})^{n-1}}\\&&\\&\ge & \alpha \cdot 2^{(\frac
{1}{t_1-\eps})^{\alpha(n-1)}}\ge \alpha\cdot 2^{(\log_p|
G/G_n|)^\alpha}.
\end{array}$$
Let $m=p|G/G_n|$. Note that any pro-$p$ group $K$ contains
$p^{d(K)}-1$ subgroups of index $p$ since $K/[K,K]K^p\cong
(\mathbb Z/p\mathbb Z)^{d(K)}$. Therefore, if we count the number
of subgroups of index $p$ in $G_n$, we obtain that $a_m(G)\ge
p^{\alpha\cdot 2^{(\log_p m -1)^\alpha}}$. This finishes the proof.
\end{proof}

\affiliationone{
   Mikhail Ershov\\
   University Of Virginia\\
   Department of Mathematics\\
   P.O. Box 400137\\
   Charlottesville, VA 22904-4137\\
   United States of America\\
   \email{ershov@virginia.edu}}
   \affiliationtwo{
   Andrei Jaikin-Zapirain\\
   Departamento de Matem\'aticas UAM\\ 
   Instituto de Ciencias Matem\'aticas, CSIC-UAM-UC3M-UCM\\
   Cantoblanco Universidad\\ 
   Madrid 28049, Spain\\
   \email{andrei.jaikin@uam.es}}

\end{document}